\documentclass[11pt]{article}

\usepackage{amsmath}
\usepackage{amssymb,amsthm,bm,url,paralist,tabls,hyperref,amscd}
\usepackage[usenames,dvipsnames]{color}
\usepackage{tikz}

\newtheorem{thm}[equation]{Theorem}
\newtheorem{prop}[equation]{Proposition}

\newtheorem{cor}[equation]{Corollary}

\newtheorem{rem}[equation]{Remark}
\newtheorem{examp}[equation]{Example}
\theoremstyle{remark}	  
\theoremstyle{definition} 

\newtheoremstyle{efronremark}
{6pt}{6pt}{}{}{\itshape}{\quad}{ }{\thmnote{#3}}

\theoremstyle{efronremark}   \newtheorem*{eremark}{}

\numberwithin{equation}{section}
\usepackage{varioref}
\labelformat{section}{Section~#1}
\labelformat{subsection}{Section~#1}
\labelformat{subsubsection}{Section~#1}
\labelformat{figure}{Figure~#1}
\labelformat{table}{Table~#1}
\labelformat{defin}{Definition~#1}
\usepackage[letterpaper,margin=1in]{geometry}

\title{Gelfand Models for Diagram Algebras}
\author{Tom Halverson\footnote{Supported in part by NSF Grant DMS-0800085.}
 \\  Department of Mathematics \\  Macalester College \\   St. Paul, MN 55105\\
 \texttt{halverson@macalester.edu}
\and
Mike Reeks\footnote{Supported in part by the Macalester College Anderson-Grossheusch  Summer Research Fund.} \\  Department of Mathematics \\   Macalester College \\   St. Paul, MN 55105\\
\texttt{mreeks@macalester.edu}}

\DeclareMathOperator{\End}{End}

\DeclareMathOperator{\Ind}{Ind}

\DeclareMathOperator{\sign}{\mathsf{sign}}

\DeclareMathOperator{\odd}{\mathrm{odd}}
\newcommand{\tldim}[2]{\left\{\begin{array}{c}{#1} \\ {#2} \end{array} \right\} }

\date{May 20, 2014}

\begin{document}

\newcommand{\s}{\mathsf{s}}
\renewcommand{\r}{\mathsf{r}}
\newcommand{\rb}{\mathsf{rb}}
\newcommand{\tl}{\mathsf{tl}}
\newcommand{\m}{\mathsf{m}}
\newcommand{\pr}{\mathsf{pr}}

\newcommand{\pn}{\mathsf{rank}}
\newcommand{\sgn}{S}
\newcommand{\KV}{\text{KV}}
\newcommand{\CC}{\mathbb{K}}
\newcommand{\ZZ}{\mathbb{Z}}
\newcommand{\V}{\mathsf{V}}
\newcommand{\G}{\mathsf{G}}
\newcommand{\A}{\mathsf{A}}

\newcommand{\U}{\mathsf{U}}
\newcommand{\vv}{\mathsf{v}}
\newcommand{\w}{\mathsf{w}}
\newcommand{\e}{\mathsf{e}}
\newcommand{\p}{\mathsf{p}}
\newcommand{\q}{\mathsf{q}}
\renewcommand{\b}{\mathsf{b}}
\newcommand{\f}{\mathsf{f}}
\newcommand{\Z}{\mathsf{Z}}

\newcommand{\M}{\mathsf{M}}
\newcommand{\W}{\mathsf{M}}
\newcommand{\Motz}{\mathsf{M}}

\newcommand{\RB}{\mathsf{RB}}
\newcommand{\RBc}{\mathcal{R\! B}}
\newcommand{\GL}{\mathsf{GL}}
\newcommand{\TL}{\mathsf{TL}}
\newcommand{\T}{\mathsf{T}}
\newcommand{\J}{\mathsf{J}}
\renewcommand{\P}{\mathsf{P}}
\newcommand{\planarP}{\mathsf{PP}}
\newcommand{\PR}{\mathsf{PR}}
\newcommand{\B}{\mathsf{B}}
\newcommand{\R}{\mathsf{R}}
\renewcommand{\H}{\mathsf{H}}

\renewcommand{\AA}{\mathcal{A}}
\newcommand{\PP}{\mathcal{P}}
\newcommand{\DD}{\mathcal{D}}
\newcommand{\BB}{\mathcal{B}}
\newcommand{\RRB}{\mathcal{RB}}
\newcommand{\RR}{\mathcal{R}}
\newcommand{\PPR}{\mathcal{P\!R}}
\newcommand{\TTL}{\mathcal{T\!L}}
\newcommand{\MM}{\mathcal{M}}

\renewcommand{\S}{\mathsf{S}}
\renewcommand{\Im}{\mathrm{Im}}
\newcommand{\rank}{\mathrm{rank}}
\newcommand{\mult}{\mathrm{mult}}
\newcommand{\Card}{\mathrm{Card}}
\def\ot{\otimes}
\def\id{{\rm id}}

\newcommand{\C}{\mathsf{C}}
\newcommand{\D}{\mathsf{D}}
\renewcommand{\T}{\mathsf{T}}
\renewcommand{\O}{\mathsf{O}}
\newcommand{\I}{\mathsf{I}}
\maketitle

\begin{abstract}
\noindent
A Gelfand model for a semisimple algebra $\mathsf{A}$ over an algebraically closed field $\mathbb{K}$ is a linear representation that contains each irreducible representation of $\mathsf{A}$ with multiplicity exactly one.  We give a method of constructing these models that works uniformly for a large class of semisimple, combinatorial diagram algebras including the partition,  Brauer,  rook monoid,  rook-Brauer,   Temperley-Lieb,  Motzkin, and  planar rook monoid algebras.  In each case, the model representation is given by diagrams acting via ``signed conjugation" on the linear span of their horizontally symmetric diagrams. This representation is a generalization of the Saxl model for  the symmetric group. Our method is to use the Jones basic construction to lift the Saxl model from the symmetric group to each diagram algebra.  In the case of the planar diagram algebras, our construction exactly produces the irreducible representations of the algebra.  \end{abstract}

\begin{eremark}[Keywords:]  Gelfand model; multiplicity-free representation;  symmetric group;  partition algebra;  Brauer algebra; Temperley-Lieb algebra; Motzkin algebra;  rook monoid.
\end{eremark}

\section*{Introduction}

A famous consequence of Robinson-Schensted-Knuth (RSK) insertion is that the set of standard Young tableaux with $k$ boxes is in bijection with the set of involutions in the symmetric group $\S_k$ (the permutations $\sigma \in \S_k$ with $\sigma^2 = 1$).  Since the standard Young tableaux index the bases for the irreducible $\S_k$ modules,  it follows that the sum of the dimensions of the irreducible $\S_k$ modules equals the number of involutions in $\S_k$. This suggests the possibility of a representation of the symmetric group on the linear span of its involutions which decomposes into irreducible $\S_k$ modules each with multiplicity  1.   Saxl  \cite{Sxl}  and  Klja\v{c}ko \cite{Klj} have constructed such a module under which the symmetric group acts on its involutions by a twisted, or signed, conjugation \eqref{SaxlSign}.   A combinatorial construction of an analogous module  was studied recently by Adin, Postnikov, and Roichman \cite{APR} and extended to the rook monoid and  related semigroups in \cite{KM}.   A representation for which each irreducible appears with multiplicity one is called a \emph{Gelfand model} (or, simply, a \emph{model}), because of the work in \cite{BGG} on models for complex Lie groups.

In \cite{HL} the RSK algorithm is extended to work for a large class of well-known, combinatorial  diagram algebras which are subalgebras of the the partition algebra. A consequence  \cite[(5.5)]{HL} of this algorithm is that the sum of the  degrees of the irreducible representations of each of these algebras equals the number of horizontally symmetric basis diagrams in the algebra.  This suggests the existence of a  model representation of each of these  algebras on the span of its symmetric diagrams, and the main result of this paper is to produce such a model.

Let $\CC$ be an algebraically closed field, and let $\A_k$ denote one of the following unital, associative $\CC$-algebras:  the partition,  Brauer,  rook monoid,  rook-Brauer,   Temperley-Lieb,  Motzkin, or  planar rook monoid algebra.
Then $\A_k$ has a basis of diagrams  and a multiplication given by diagram concatenation.   
The algebra $\A_k$ depends on a parameter $x \in \CC$ and is semisimple when $char(\CC) = 0$ or  under special conditions on $char(\CC)>0$ and for all but a finite number of choices of $x \in \CC$.  When $\A_k$ is semisimple, its irreducible modules are indexed by a set $\Lambda_{\A_k}$, and for $\lambda \in \Lambda_{\A_k}$, we let $\A_k^\lambda$ denote the irreducible $\A_k$-module labeled by $\lambda$.  We construct, in a uniform way, an $\A_k$-module $\M_{\A_k}$ which decomposes  as
$
\M_{\A_k} \cong \bigoplus_{\lambda \in \Lambda_{\A_k}} \A_k^\lambda.
$ 

Our model representation is constructed as follows. 
For a basis diagram $d$,  let $d^T$ be its reflection across its horizontal axis and  say that a diagram $t$ is symmetric if  $t^T = t$.  A basis diagram $d$ acts on a symmetric diagram $t$ by ``signed conjugation": $d \cdot t = \sgn(d,t) \, d t d^T$, where $\sgn(d,t)$ is the sign on the permutation of the fixed blocks of $t$ induced by conjugation by $d$ (see Example \ref{SignedConjugation} for details).   In each example, our basis diagrams are assigned a rank, which is the number of blocks in the diagram that propagate from the top row to the bottom row.  We let $\M_{\A_k}^{r}$ be the linear span of the symmetric diagrams of rank $r$ and  our  model is the direct sum $\M_{\A_k} = \oplus_{r=0}^k \M_{\A_k}^r$.

The diagram algebras in this paper naturally form a tower $\A_0 \subseteq \A_1  \subseteq \cdots \subseteq \A_k$. Each algebra contains a Jones basic construction ideal $\J_{k-1} \subseteq \A_k$ such that
$
\A_k \cong \J_{k-1} \oplus \C_k,
$
where $\C_k \cong \CC \S_k$ for nonplanar diagram algebras and  $\C_k \cong \CC {\bf 1}_k$ for  planar diagram algebras.   The ideal $\J_{k-1}$ is in Schur-Weyl duality with one of $\A_{k-1}$ or $\A_{k-2}$ (depending on the specific diagram algebra), and we  are able to take models for each $\C_r, 0 \le r \le k$, and lift them to a  model for $\A_k$.

For the planar diagram algebras --- the Temperley-Lieb, Motzkin, and planar rook monoid algebras --- the algebra $\C \cong \CC {\bf 1}_k$ is trivial and  the model  is trivial. It follows  that $\M_{\A_k}^{r}$ is irreducible and that  signed conjugation  produces a complete set of irreducible modules for the planar algebras. For the nonplanar diagram algebras, the algebra is $\C \cong \CC \S_k$, and we use  the Saxl model for $\S_k$.  In this case  $\M_{\A_k}^r$ is further graded as $\M_{\A_k}^r = \oplus_f  \M_{\A_k}^{r,f}$, where $\M_{\A_k}^{r,f}$ is the linear span of symmetric diagrams of rank $r$ having $f$ ``fixed blocks,"  and $\M_{\A_k}^{r,f}$ decomposes into irreducibles labeled by partitions having $f$ odd parts.

Besides being natural constructions, these  representations are useful in several ways:
(1) In a model representation,  isotypic  components are  irreducible components, so projection operators map  directly onto  irreducible modules without being mixed up among multiple isomorphic copies of the same module.  (2) A key feature of our model is that we give the explicit action of each basis element of $\A_k$ on the basis of $\M_{\A_k}^{r,f}$.  For small values of $k$, and for all values of $k$ in the planar case, these representations  are  irreducible or have few irreducible components. Thus, in practice, the model provides a natural and easy way to compute the explicit action of basis diagrams on irreducible representations.  Indeed, it is through this construction that the irreducible modules for the Motzkin \cite{BH}, the rook-Brauer \cite{dH}, and the planar rook monoid \cite{FHH} were discovered.   (3)
Gelfand models are useful in the study of Markov chains on related combinatorial objects; see, for example, Chapter 3F of \cite{Di} and the references therein, as well as \cite{DH}, \cite{RSW}.

Finally, the enumeration of symmetric diagrams in these algebras according to rank and number of fixed blocks gives rise to well-known, interesting integer sequences.  These combinatorics are analyzed in \ref{ModelDiags}, where we work out the  details of the model representation for each   algebra.

\bigskip
\noindent
\textbf{Acknowledgements}.  We thank Arun Ram for suggesting that we look for model representations of these algebras after seeing the dimension results in \cite{HL}.  We also thank Michael Decker \cite{Dec}, whose honors project, under the direction T.\ Halverson, examined the model characters  of the symmetric group and the partition algebra.  It was during this collaboration that we constructed the combinatorial Saxl model for the symmetric group and conjectured the general construction of Gelfand models for diagram algebras.  Upon the completion and submission of this manuscript, we learned of the preprint by V.\ Mazourchuk \cite{Mz}, who uses different methods to derive an analogous  model to the one in this paper (see the comments in \ref{ModelComparisonSection}).   We also thank the anonymous referees for several helpful recommendations. 
\begin{section}{The Partition Algebra and its Diagram Subalgebras}
\label{sec:DiagAlgs}

In this section, we describe the partition  algebra $\P_k(x)$ over $\CC$ with a parameter $x \in \CC$ and realize the other diagram algebras of interest in this paper as  subalgebras of $\P_k(x)$.  The main results of this paper require that $\CC$ be chosen so that $\P_k(n)$ is semisimple. For example one may choose $\CC$ such that $char(\CC) = 0$.

\begin{subsection}{The partition monoid $\PP_k$}\label{sec:PartMonoid}

For $k \in \ZZ_{> 0}$, let $\PP_k$ denote the set of set partitions of $\{1, 2, \ldots, k, 1', 2', \ldots, k'\}$.  We  represent a set partition $d\in \PP_k$ by a diagram with vertices in the top row labeled $1, \ldots,k$  and vertices
in the bottom row labeled $1', \ldots, k'$.  Assign edges in this diagram so that the connected components equal the underlying set partition $d$. For example, the following is a diagram  $d \in \PP_{12}$,
$$
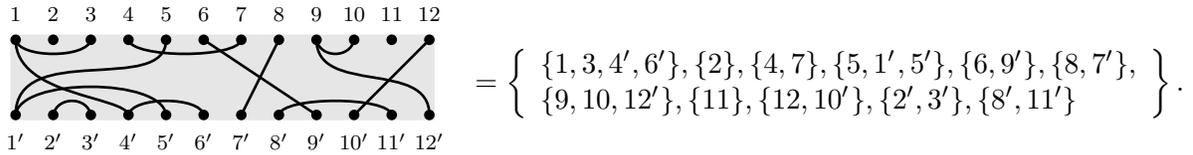
\begin{array}{c}
\begin{tikzpicture}[scale=.5,line width=1pt] 
\foreach \i in {1,...,12} 
{ \path (\i,1) coordinate (T\i); \path (\i,-1) coordinate (B\i); } 
\filldraw[fill= black!10,draw=black!10,line width=4pt]  (T1) -- (T12) -- (B12) -- (B1) -- (T1);
\draw[black] (T1)  .. controls +(.1,-.5) and +(-.1,-.5) .. (T3) ;
\draw[black] (T9)  .. controls +(.1,-.5) and +(-.1,-.5) .. (T10) ;
\draw[black] (T4) .. controls +(.1,-.5) and +(-.1,-.5) .. (T7) ;
\draw[black] (T1) .. controls +(0,-1.5) and +(-.1,.5) .. (B4);
\draw[black] (T5) .. controls +(0,-1.5) and +(0,2) .. (B1);
\draw[black] (T6) -- (B9);
\draw[black] (T8) -- (B7);
\draw[black] (T9) .. controls +(0,-1.5) and +(0,1.5) ..  (B12);
\draw[black] (T12) -- (B10);
\draw[black] (B8) .. controls +(.1,.5) and +(-.1,.5) .. (B11) ;
\draw[black] (B4) .. controls +(.1,.5) and +(-.1,.5) .. (B6) ;
\draw[black] (B1) .. controls +(.1,1) and +(-.1,1) .. (B5) ;
\draw[black] (B2) .. controls +(.1,.5) and +(-.1,.5) .. (B3) ;\foreach \i in {1,...,12} 
{ \fill (T\i) circle (4pt); \fill (B\i) circle (4pt); }
\draw  (T1)  node[above=0.1cm]{$\scriptstyle{1}$};\draw  (B1)  node[below=0.1cm]{$\scriptstyle{1'}$};
\draw  (T2)  node[above=0.1cm]{$\scriptstyle{2}$};\draw  (B2)  node[below=0.1cm]{$\scriptstyle{2'}$};
\draw  (T3)  node[above=0.1cm]{$\scriptstyle{3}$};\draw  (B3)  node[below=0.1cm]{$\scriptstyle{3'}$};
\draw  (T4)  node[above=0.1cm]{$\scriptstyle{4}$};\draw  (B4)  node[below=0.1cm]{$\scriptstyle{4'}$};
\draw  (T5)  node[above=0.1cm]{$\scriptstyle{5}$};\draw  (B5)  node[below=0.1cm]{$\scriptstyle{5'}$};
\draw  (T6)  node[above=0.1cm]{$\scriptstyle{6}$};\draw  (B6)  node[below=0.1cm]{$\scriptstyle{6'}$};
\draw  (T7)  node[above=0.1cm]{$\scriptstyle{7}$};\draw  (B7)  node[below=0.1cm]{$\scriptstyle{7'}$};
\draw  (T8)  node[above=0.1cm]{$\scriptstyle{8}$};\draw  (B8)  node[below=0.1cm]{$\scriptstyle{8'}$};
\draw  (T9)  node[above=0.1cm]{$\scriptstyle{9}$};\draw  (B9)  node[below=0.1cm]{$\scriptstyle{9'}$};
\draw  (T10)  node[above=0.1cm]{$\scriptstyle{10}$};\draw  (B10)  node[below=0.1cm]{$\scriptstyle{10'}$};
\draw  (T11)  node[above=0.1cm]{$\scriptstyle{11}$};\draw  (B11)  node[below=0.1cm]{$\scriptstyle{11'}$};
\draw  (T12)  node[above=0.1cm]{$\scriptstyle{12}$};\draw  (B12)  node[below=0.1cm]{$\scriptstyle{12'}$};
\end{tikzpicture} 
\end{array}
 = 
\left\{
\begin{array}{l}
\{1,3,4',6'\}, \{2\}, \{4,7\}, \{5,1',5'\},  \{6, 9'\},  \{8,7'\}, \\
  \{9,10,12'\}, \{11\}, \{12,10'\}, \{2',3'\}, \{8',11'\} 
\end{array}
\right\}.
$$
We refer to the parts of a set partition as \emph{blocks}, so that the above diagram has 11 blocks.
The diagram of $d$ is not unique, since it only depends on the underlying  connected components.  We make the following convention:  if a block contains vertices from both the top row and bottom row, then \emph{we  always connect the leftmost vertex in the top row of a block with the leftmost vertex in the bottom row of the block by a single vertical edge}.

 We multiply two set partition diagrams $d_1, d_2 \in \PP_k$ as follows. Place $d_1$ above $d_2$ and identify each vertex $j'$ in the bottom row of $d_1$ with the corresponding vertex $j$ in the top row of $d_2$. Remove any connected components that live entirely in the middle row and let $d_1 \circ d_2 \in \PP_k$  be the resulting diagram.  For example, if 
$$
d_1 = 
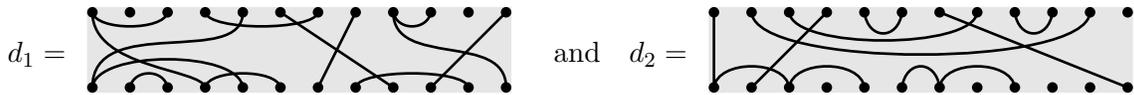
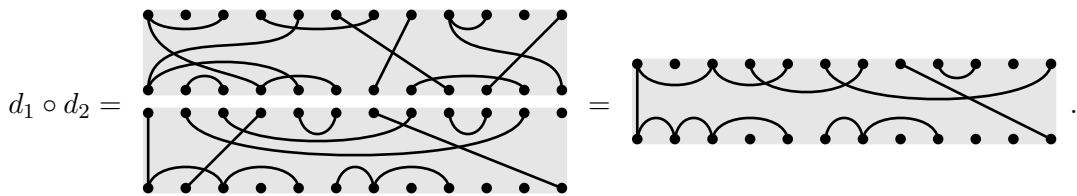
\begin{array}{c}
\begin{tikzpicture}[scale=.5,line width=1pt] 
\foreach \i in {1,...,12} 
{ \path (\i,1) coordinate (T\i); \path (\i,-1) coordinate (B\i); } 
\filldraw[fill= black!10,draw=black!10,line width=4pt]  (T1) -- (T12) -- (B12) -- (B1) -- (T1);
\draw[black] (T1)  .. controls +(.1,-.5) and +(-.1,-.5) .. (T3) ;
\draw[black] (T9)  .. controls +(.1,-.5) and +(-.1,-.5) .. (T10) ;
\draw[black] (T4) .. controls +(.1,-.5) and +(-.1,-.5) .. (T7) ;
\draw[black] (T1) .. controls +(0,-1.5) and +(-.1,.5) .. (B4);
\draw[black] (T5) .. controls +(0,-1.5) and +(0,2) .. (B1);
\draw[black] (T6) -- (B9);
\draw[black] (T8) -- (B7);
\draw[black] (T9) .. controls +(0,-1.5) and +(0,1.5) ..  (B12);
\draw[black] (T12) -- (B10);
\draw[black] (B8) .. controls +(.1,.5) and +(-.1,.5) .. (B11) ;
\draw[black] (B4) .. controls +(.1,.5) and +(-.1,.5) .. (B6) ;
\draw[black] (B1) .. controls +(.1,1) and +(-.1,1) .. (B5) ;
\draw[black] (B2) .. controls +(.1,.5) and +(-.1,.5) .. (B3) ;
\foreach \i in {1,...,12} 
{ \fill (T\i) circle (4pt); \fill (B\i) circle (4pt); }
\end{tikzpicture} 
\end{array}
\quad\hbox{and}\quad
d_2 = 
\begin{array}{c}
\begin{tikzpicture}[scale=.5,line width=1pt] 
\foreach \i in {1,...,12} 
{ \path (\i,1) coordinate (T\i); \path (\i,-1) coordinate (B\i); } 
\filldraw[fill= black!10,draw=black!10,line width=4pt]  (T1) -- (T12) -- (B12) -- (B1) -- (T1);
\draw[black] (T3)  .. controls +(.1,-1) and +(-.1,-1) .. (T8) ;
\draw[black] (T2)  .. controls +(.1,-1.5) and +(-.1,-1.5) .. (T11) ;
\draw[black] (T9)  .. controls +(.1,-.75) and +(-.1,-.75) .. (T10) ;
\draw[black] (T5)  .. controls +(.1,-.75) and +(-.1,-.75) .. (T6) ;
\draw[black] (T1) -- (B1);
\draw[black] (T4) -- (B2);
\draw[black] (T7) -- (B12);
\draw[black] (B1) .. controls +(.1,.75) and +(-.1,.75) .. (B3) ;
\draw[black] (B3) .. controls +(.1,.75) and +(-.1,.75) .. (B5) ;
\draw[black] (B6) .. controls +(.1,.75) and +(-.1,.75) .. (B7) ;
\draw[black] (B7) .. controls +(.1,.75) and +(-.1,.75) .. (B9) ;

\foreach \i in {1,...,12} 
{ \fill (T\i) circle (4pt); \fill (B\i) circle (4pt); }
\end{tikzpicture} 
\end{array}
$$
then 
$$
d_1 \circ d_2 = 
 \begin{array}{c}
\begin{tikzpicture}[scale=.5,line width=1pt] 
\foreach \i in {1,...,12} 
{ \path (\i,1) coordinate (T\i); \path (\i,-1) coordinate (B\i); } 
\filldraw[fill= black!10,draw=black!10,line width=4pt]  (T1) -- (T12) -- (B12) -- (B1) -- (T1);
\draw[black] (T1)  .. controls +(.1,-.5) and +(-.1,-.5) .. (T3) ;
\draw[black] (T9)  .. controls +(.1,-.5) and +(-.1,-.5) .. (T10) ;
\draw[black] (T4) .. controls +(.1,-.5) and +(-.1,-.5) .. (T7) ;
\draw[black] (T1) .. controls +(0,-1.5) and +(-.1,.5) .. (B4);
\draw[black] (T5) .. controls +(0,-1.5) and +(0,2) .. (B1);
\draw[black] (T6) -- (B9);
\draw[black] (T8) -- (B7);
\draw[black] (T9) .. controls +(0,-1.5) and +(0,1.5) ..  (B12);
\draw[black] (T12) -- (B10);
\draw[black] (B8) .. controls +(.1,.5) and +(-.1,.5) .. (B11) ;
\draw[black] (B4) .. controls +(.1,.5) and +(-.1,.5) .. (B6) ;
\draw[black] (B1) .. controls +(.1,1) and +(-.1,1) .. (B5) ;
\draw[black] (B2) .. controls +(.1,.5) and +(-.1,.5) .. (B3) ;
\foreach \i in {1,...,12} 
{ \fill (T\i) circle (4pt); \fill (B\i) circle (4pt); }
\end{tikzpicture}  \\
\begin{tikzpicture}[scale=.5,line width=1pt] 
\foreach \i in {1,...,12} 
{ \path (\i,1) coordinate (T\i); \path (\i,-1) coordinate (B\i); } 
\filldraw[fill= black!10,draw=black!10,line width=4pt]  (T1) -- (T12) -- (B12) -- (B1) -- (T1);
\draw[black] (T3)  .. controls +(.1,-1) and +(-.1,-1) .. (T8) ;
\draw[black] (T2)  .. controls +(.1,-1.5) and +(-.1,-1.5) .. (T11) ;
\draw[black] (T9)  .. controls +(.1,-.75) and +(-.1,-.75) .. (T10) ;
\draw[black] (T5)  .. controls +(.1,-.75) and +(-.1,-.75) .. (T6) ;
\draw[black] (T1) -- (B1);
\draw[black] (T4) -- (B2);
\draw[black] (T7) -- (B12);
\draw[black] (B1) .. controls +(.1,.75) and +(-.1,.75) .. (B3) ;
\draw[black] (B3) .. controls +(.1,.75) and +(-.1,.75) .. (B5) ;
\draw[black] (B6) .. controls +(.1,.75) and +(-.1,.75) .. (B7) ;
\draw[black] (B7) .. controls +(.1,.75) and +(-.1,.75) .. (B9) ;
\foreach \i in {1,...,12} 
{ \fill (T\i) circle (4pt); \fill (B\i) circle (4pt); }
\end{tikzpicture} \end{array}
= 
\begin{array}{c}
\begin{tikzpicture}[scale=.5,line width=1pt] 
\foreach \i in {1,...,12} 
{ \path (\i,1) coordinate (T\i); \path (\i,-1) coordinate (B\i); } 
\filldraw[fill= black!10,draw=black!10,line width=4pt]  (T1) -- (T12) -- (B12) -- (B1) -- (T1);
\draw[black] (T1)  .. controls +(.1,-.75) and +(-.1,-.75) .. (T3) ;
\draw[black] (T3)  .. controls +(.1,-.75) and +(-.1,-.75) .. (T5) ;
\draw[black] (T4) .. controls +(.1,-1) and +(-.1,-1) .. (T7) ;
\draw[black] (T9)  .. controls +(.1,-.5) and +(-.1,-.5) .. (T10) ;
\draw[black] (T6)  .. controls +(.1,-1.25) and +(-.1,-1.25) .. (T12) ;
\draw[black] (T1) -- (B1);
\draw[black] (T8) -- (B12);
\draw[black] (B1) .. controls +(.1,.75) and +(-.1,.75) .. (B2) ;
\draw[black] (B2) .. controls +(.1,.75) and +(-.1,.75) .. (B3) ;
\draw[black] (B3) .. controls +(.1,.75) and +(-.1,.75) .. (B5) ;
\draw[black] (B6) .. controls +(.1,.75) and +(-.1,.75) .. (B7) ;
\draw[black] (B7) .. controls +(.1,.75) and +(-.1,.75) .. (B9) ;
\foreach \i in {1,...,12} 
{ \fill (T\i) circle (4pt); \fill (B\i) circle (4pt); }
\end{tikzpicture} \end{array}.
$$
Diagram multiplication is associative and makes $\PP_k$ a monoid with identity
$
\mathbf{1}_k = 
\begin{array}{c}
\begin{tikzpicture}[scale=.4,line width=1pt] 
\foreach \i in {1,...,6} 
{ \path (\i,.5) coordinate (T\i); \path (\i,-.5) coordinate (B\i); } 
\filldraw[fill= black!10,draw=black!10,line width=4pt]  (T1) -- (T6) -- (B6) -- (B1) -- (T1);
\draw[black] (T1) -- (B1);
\draw[black] (T2) -- (B2);
\draw[black] (T3) -- (B3);
\draw[black] (T5) -- (B5);
\draw[black] (T6) -- (B6);
\foreach \i in {1,...,3} 
{ \fill (T\i) circle (4pt); \fill (B\i) circle (4pt); } 
\foreach \i in {5,...,6} 
{ \fill (T\i) circle (4pt); \fill (B\i) circle (4pt); } 
\draw (T4) node {$\cdots$};
\draw (B4) node {$\cdots$};
\end{tikzpicture}\end{array}.
$
\end{subsection}

\begin{subsection}{The partition algebra $\P_k(x)$}

Now let $x \in \CC$, define  $\P_0(x) = \CC$, and for $k \geq 1$,  let $\P_k(x)$ be the $\CC$-vector space with basis  $\PP_k$.  If $d_1, d_2 \in \PP_k$,  let  $\kappa(d_1,d_2)$ denote the number of  connected components that are removed from the middle row in computing $d_1 \circ d_2$, and define
\begin{equation}\label{PartitionAlgebraMultiplication}
d_1 d_2 = x^{\kappa(d_1,d_2)}\, d_1 \circ d_2.
\end{equation}
In the multiplication example of the previous section $\kappa(d_1,d_2) = 1$ and  $d_1 d_2 = x (d_1 \circ d_2)$.
This product makes $\P_k(x)$ an associative algebra with identity ${\bf 1}_k$.

We say that a block $B$ in a set partition diagram $d \in \PP_k$ is a \emph{propagating} block if $B$ contains vertices from both the top and bottom row of $d$; that is, both $B \cap \{1, 2, \ldots, k\}$ and $B \cap \{1',2',\ldots,k'\}$ are nonempty. 
The \emph{rank} of $d \in  \PP_k$ (also called the \emph{propagating number}) is
\begin{equation}\label{def:propagating}
\pn(d) = 
\left(
\begin{array}{l}
\hbox{the number of  propagating blocks in $d$}
\end{array}\right).
\end{equation}
The rank satisfies 
\begin{equation}
\pn(d_1 d_2) \le \min(\pn(d_1),\pn(d_2)). 
\end{equation}
For $0 \le r \le k$, we let $\J_r \subseteq \P_k(x)$ be the $\CC$-span of the diagrams of rank \emph{less than or equal} to $r$.   Then $\J_r$ is a two-sided ideal in $\P_k(x)$, and we have a tower of ideals:
$\J_0 \subseteq \J_1 \subseteq \J_2 \subseteq \cdots \subseteq \J_k  = \P_k(x).$

The partition algebra was first defined independently by P.P. Martin \cite{Ma} and V.F.R.\ Jones \cite{Jo2} as a higher-dimension generalization of the Temperley-Lieb algebra in statistical mechanics (see also \cite{HR2} for a survey of many results on the partition algebra).

\end{subsection}

\begin{subsection}{Subalgebras}\label{sec:subalgebras}

For each $k \in  \ZZ_{>0}$, the following are subalgebras of the partition algebra $\P_k(x)$:
\begin{eqnarray*}
\CC\S_k & =&   \CC\hbox{-span}\{\, d \in \PP_k \ | \ \pn(d) = k \}, \\
\B_k(x) & =&  \CC\hbox{-span}\{\, d\in \PP_k \ |\ \hbox{all blocks of $d$ have size 2}\}, \\
\R_k &=& \CC\hbox{-span}\left\{\, d\in \PP_k \ \bigg|\ 
\begin{array}{l}
\hbox{all blocks of $d$ have at most one vertex in $\{1, \ldots k\}$} \\
\hbox{and at most one vertex in $\{1', \ldots k'\}$} \\
\end{array}\right\}, \\
\RB_k(x) &=& \CC\hbox{-span}\{\, d\in \PP_k \ |\ \hbox{all blocks of $d$ have size 1 or 2}\}.
\end{eqnarray*}
Here, $\CC\S_k$ is the group algebra of the symmetric group, $\B_k(x)$ is the Brauer algebra \cite{Br},  $\R_k$ is the rook monoid algebra \cite{So}, and $\RB_k(x)$ is the rook-Brauer algebra \cite{dH}, \cite{MM}.  

A set partition is {\it planar}  if it can be represented as a diagram without edge crossings inside of the rectangle formed
by its vertices.  The planar partition algebra \cite{Jo2} is
$$
\planarP_k(x) =  \CC\hbox{-span}\{\, d \in \PP_k \ | \ d \hbox{ is planar }  \}.
$$
The following are the planar subalgebras of $\P_k(x)$:
$$
\begin{array}{rclcrcl}
\CC\{ {\bf 1}_k \} & =&  \CC \S_k \cap \planarP_k(x), & \hskip.4in &  \TL_k(x) & = &   \B_k(x) \cap \planarP_k(x), \\
\PR_k & =&   \R_k \cap \planarP_k(x), && \Motz_{k}(x) & = &  \RB_k(x) \cap \planarP_k(x).
\end{array}
$$
Here, $\TL_k(x)$ is the Temperley-Lieb algebra \cite{TL}, $\PR_k$ is the planar rook monoid algebra \cite{FHH}, and $\Motz_{k}(x)$ is the Motzkin algebra \cite{BH}.  The parameter $x$ does not arise when multiplying symmetric group diagrams (as there are never middle blocks to be removed).  The parameter is set to be $x =1$ for the rook monoid algebra and the planar rook monoid algebra. Here are examples from each of these subalgebras:
$$
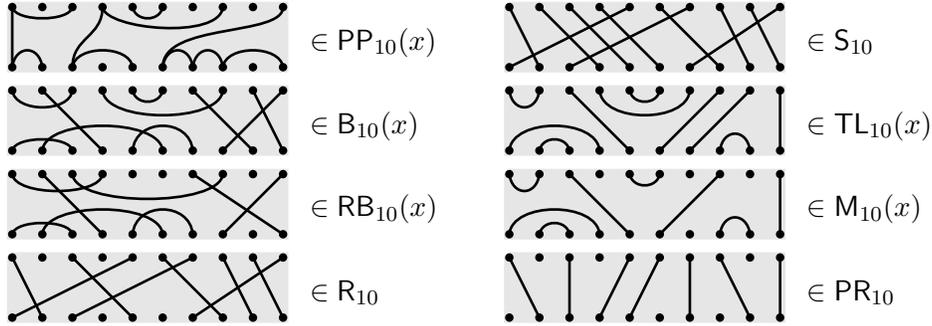
\begin{array}{l c l}
\begin{array}{c}
\begin{tikzpicture}[scale=.4,line width=1pt] 
\foreach \i in {1,...,10}  { \path (\i,1) coordinate (T\i); \path (\i,-1) coordinate (B\i); } 
\filldraw[fill= black!10,draw=black!10,line width=4pt]  (T1) -- (T10) -- (B10) -- (B1) -- (T1);
\draw[black] (T1) .. controls +(.1,-.5) and +(-.1,-.5) .. (T3) ;
\draw[black] (T4) .. controls +(.1,-.75) and +(-.1,-1.1) .. (T8) ;
\draw[black] (T5) .. controls +(.1,-.5) and +(-.1,-.5) .. (T6) ;
\draw[black] (T1) -- (B1);
\draw[black] (T4) .. controls +(0,-1) and +(0,1) .. (B3);
\draw[black] (T10) .. controls +(0,-1) and +(0,1.5) .. (B6);
\draw[black] (B1) .. controls +(.1,.75) and +(-.1,.75) .. (B2) ;
\draw[black] (B3) .. controls +(.1,.75) and +(-.1,.75) .. (B5) ;
\draw[black] (B6) .. controls +(.1,.75) and +(-.1,.75) .. (B7) ;
\draw[black] (B7) .. controls +(.1,.75) and +(-.1,.75) .. (B8) ;
\draw[black] (B8) .. controls +(.1,.75) and +(-.1,.75) .. (B10) ;
\foreach \i in {1,...,10}  { \fill (T\i) circle (4pt); \fill (B\i) circle (4pt); } 
\end{tikzpicture}\end{array}
 \in \planarP_{10}(x) &  & 
 \begin{array}{c}
 \begin{tikzpicture}[scale=.4,line width=1pt] 
\foreach \i in {1,...,10}  { \path (\i,1) coordinate (T\i); \path (\i,-1) coordinate (B\i); } 
\filldraw[fill= black!10,draw=black!10,line width=4pt]  (T1) -- (T10) -- (B10) -- (B1) -- (T1);
\draw[black] (T1) -- (B2);\
\draw[black] (T2) -- (B4);
\draw[black] (T3) -- (B5);
\draw[black] (T4) -- (B6);
\draw[black] (T5) -- (B1);
\draw[black] (T6) -- (B8);
\draw[black] (T7) -- (B3);
\draw[black] (T8) -- (B9);
\draw[black] (T9) -- (B10);
\draw[black] (T10) -- (B7);
\foreach \i in {1,...,10}  { \fill (T\i) circle (4pt); \fill (B\i) circle (4pt); } 
\end{tikzpicture}\end{array}
 \in \S_{10} 
 \\
\begin{array}{c}
\begin{tikzpicture}[scale=.4,line width=1pt] 
\foreach \i in {1,...,10}  { \path (\i,1) coordinate (T\i); \path (\i,-1) coordinate (B\i); } 
\filldraw[fill= black!10,draw=black!10,line width=4pt]  (T1) -- (T10) -- (B10) -- (B1) -- (T1);
\draw[black] (T7) -- (B9);
\draw[black] (T10) -- (B8);
\draw[black] (T9) -- (B10);
\draw[black] (T2) -- (B4);
\draw[black] (T1) .. controls +(.1,-.75) and +(-.1,-.75) .. (T3) ;
\draw[black] (T4) .. controls +(.1,-1.1) and +(-.1,-1.1) .. (T8) ;
\draw[black] (T5) .. controls +(.1,-.5) and +(-.1,-.5) .. (T6) ;
\draw[black] (B1) .. controls +(.1,.5) and +(-.1,.5) .. (B3) ;
\draw[black] (B5) .. controls +(.1,1.1) and +(-.1,1.1) .. (B7) ;
\draw[black] (B2) .. controls +(.1,1.1) and +(-.1,1.1) .. (B6) ;
\foreach \i in {1,...,10}  { \fill (T\i) circle (4pt); \fill (B\i) circle (4pt); } 
\end{tikzpicture}\end{array}
 \in \B_{10}(x) &  & 
 \begin{array}{c}
 \begin{tikzpicture}[scale=.4,line width=1pt] 
\foreach \i in {1,...,10}  { \path (\i,1) coordinate (T\i); \path (\i,-1) coordinate (B\i); } 
\filldraw[fill= black!10,draw=black!10,line width=4pt]  (T1) -- (T10) -- (B10) -- (B1) -- (T1);
\draw[black] (T8) -- (B6);
\draw[black] (T10) -- (B10);
\draw[black] (T9) -- (B7);
\draw[black] (T3) -- (B5);
\draw[black] (T1) .. controls +(.1,-.75) and +(-.1,-.75) .. (T2) ;
\draw[black] (T4) .. controls +(.1,-1.1) and +(-.1,-1.1) .. (T7) ;
\draw[black] (T5) .. controls +(.1,-.5) and +(-.1,-.5) .. (T6) ;
\draw[black] (B2) .. controls +(.1,.5) and +(-.1,.5) .. (B3) ;
\draw[black] (B8) .. controls +(.1,.75) and +(-.1,.75) .. (B9) ;
\draw[black] (B1) .. controls +(.1,1.1) and +(-.1,1.1) .. (B4) ;
\foreach \i in {1,...,10}  { \fill (T\i) circle (4pt); \fill (B\i) circle (4pt); } 
\end{tikzpicture}\end{array}
 \in \mathsf{TL}_{10}(x)
 \\
 \begin{array}{c}
\begin{tikzpicture}[scale=.4,line width=1pt] 
\foreach \i in {1,...,10}  { \path (\i,1) coordinate (T\i); \path (\i,-1) coordinate (B\i); } 
\filldraw[fill= black!10,draw=black!10,line width=4pt]  (T1) -- (T10) -- (B10) -- (B1) -- (T1);
\draw[black] (T7) -- (B10);
\draw[black] (T10) -- (B8);
\draw[black] (T2) -- (B4);
\draw[black] (T1) .. controls +(.1,-.75) and +(-.1,-.75) .. (T4) ;
\draw[black] (T3) .. controls +(.1,-1.1) and +(-.1,-1.1) .. (T8) ;
\draw[black] (B1) .. controls +(.1,.5) and +(-.1,.5) .. (B3) ;
\draw[black] (B5) .. controls +(.1,1.1) and +(-.1,1.1) .. (B7) ;
\draw[black] (B2) .. controls +(.1,1.1) and +(-.1,1.1) .. (B6) ;
\foreach \i in {1,...,10}  { \fill (T\i) circle (4pt); \fill (B\i) circle (4pt); } 
\end{tikzpicture}\end{array}
 \in \mathsf{RB}_{10}(x)   &  & 
 \begin{array}{c}
\begin{tikzpicture}[scale=.4,line width=1pt] 
\foreach \i in {1,...,10}  { \path (\i,1) coordinate (T\i); \path (\i,-1) coordinate (B\i); } 
\filldraw[fill= black!10,draw=black!10,line width=4pt]  (T1) -- (T10) -- (B10) -- (B1) -- (T1);
\draw[black] (T8) -- (B6);
\draw[black] (T10) -- (B10);
\draw[black] (T3) -- (B5);
\draw[black] (T1) .. controls +(.1,-.75) and +(-.1,-.75) .. (T2) ;
\draw[black] (T5) .. controls +(.1,-.5) and +(-.1,-.5) .. (T6) ;
\draw[black] (B2) .. controls +(.1,.5) and +(-.1,.5) .. (B3) ;
\draw[black] (B8) .. controls +(.1,.75) and +(-.1,.75) .. (B9) ;
\draw[black] (B1) .. controls +(.1,1.1) and +(-.1,1.1) .. (B4) ;
\foreach \i in {1,...,10}  { \fill (T\i) circle (4pt); \fill (B\i) circle (4pt); } 
\end{tikzpicture}\end{array}
 \in \Motz_{10}(x)  \\
 \begin{array}{c}
\begin{tikzpicture}[scale=.4,line width=1pt] 
\foreach \i in {1,...,10}  { \path (\i,1) coordinate (T\i); \path (\i,-1) coordinate (B\i); } 
\filldraw[fill= black!10,draw=black!10,line width=4pt]  (T1) -- (T10) -- (B10) -- (B1) -- (T1);
\draw[black] (T1) -- (B2);
\draw[black] (T3) -- (B5);
\draw[black] (T5) -- (B1);
\draw[black] (T6) -- (B8);
\draw[black] (T7) -- (B3);
\draw[black] (T8) -- (B9);
\draw[black] (T9) -- (B10);
\draw[black] (T10) -- (B7);
\foreach \i in {1,...,10}  { \fill (T\i) circle (4pt); \fill (B\i) circle (4pt); } 
\end{tikzpicture}\end{array}
 \in \R_{10} 
&  & 
 \begin{array}{c}
\begin{tikzpicture}[scale=.4,line width=1pt] 
\foreach \i in {1,...,10}  { \path (\i,1) coordinate (T\i); \path (\i,-1) coordinate (B\i); } 
\filldraw[fill= black!10,draw=black!10,line width=4pt]  (T1) -- (T10) -- (B10) -- (B1) -- (T1);
\draw[black] (T1) -- (B2);
\draw[black] (T3) -- (B3);
\draw[black] (T5) -- (B4);
\draw[black] (T6) -- (B5);
\draw[black] (T7) -- (B7);
\draw[black] (T8) -- (B9);
\draw[black] (T10) -- (B10);
\foreach \i in {1,...,10}  { \fill (T\i) circle (4pt); \fill (B\i) circle (4pt); } 
\end{tikzpicture}\end{array}
 \in \PR_{10}
\end{array}
$$

\end{subsection}

\end{section}

\begin{section}{A Model Representation of the Symmetric Group}

\begin{subsection}{Saxl's model characters of $\S_k$}\label{sec:Saxl}

An involution $t \in \S_k$ is a permutation such that $t^2 = 1$. In disjoint cycle notation, involutions consist of 2-cycles and fixed points. Let $\I_k$ be the set of involutions in $\S_k$ and let $\I_{k}^f$ be the involutions in $\S_k$ which fix precisely $f$ points; that is,
\begin{equation}
\I_{\S_k} = \left\{\ t \in \S_n \ \big\vert  \  t^2 = 1 \right\} \quad\hbox{and}\quad
\I_{\S_k}^f = \left\{\ t \in \S_n \ \big\vert  \  t^2 = 1 \text{ and $t$ has $f$ fixed points } \right\}.
\end{equation}
For a fixed involution $t \in \I_{k}^f$, let $\C(t)\subseteq \S_n$ be the centralizer of $t$ in $\S_k$. Let $\S_k$ act on itself by conjugation so that 
$w \cdot \sigma = w \sigma w^{-1}$ for all $w, \sigma \in \S_k$.  Then $\C(t)$ and  $\I_{\S_k}^{f}$ are the stabilizer  and  orbit of $t$, respectively, so $|\S_k|= |\C(t)| \cdot |\I_{\S_k}^{f}|$.  

If $w \in \C(t)$,  then $w t w^{-1} = t$, so  $w$ fixes $t$ but possibly permutes the fixed points of $t$.  
Let $\pi_{f}$ be the linear character of $\C(t)$ such that $\pi_{f}(w)$ is the sign of the permutation of $w$ on the fixed points of $t$.  Let $\odd(\lambda)$ denote the number of odd parts of the partition $\lambda$. Saxl
 \cite{Sxl}  (see also \cite{Klj} or \cite{IRS}) proves the following decomposition of the induced character
\begin{equation}\label{SaxlThm}
\varphi_{\S_k}^f := \Ind_{\C(t)}^{\S_n} (\pi_{f}) =\!\!\! \sum_{ \genfrac{}{}{0pt}{}{\lambda \vdash k}{\odd(\lambda)=f}}\!\!\! \chi^\lambda_{\S_k}
\qquad\hbox{and thus}\qquad
\varphi_{\S_k} :=  \sum_{\ell = 0}^{\lfloor k/2 \rfloor} \varphi^{k - 2 \ell}_{\S_k} = \sum_{\lambda\vdash k}  \chi^\lambda_{\S_k}.
\end{equation}
This  generalizes the classic result (see \cite[Theorem IV]{Th} or \cite[5.4.23]{JK}) for fixed-point-free permutations, i.e., the case where $f = 0$.  In this case,  there are no fixed points and  $\pi_{0}$ is the trivial character of $\C(t)$.

The number of involutions in $\S_k$ having $f = k - 2\ell$ fixed points and $\ell$ transpositions is
\begin{equation}
 |\I_{\S_k}^{f}| = |\I_{\S_k}^{k-2 \ell}|= \binom{k}{2 \ell} (2 \ell-1)!!, \qquad\hbox{ where $(2 \ell-1) !! = (2 \ell -1) (2 \ell -3) \cdots 3 \cdot 1.$ }
\end{equation}
If we let $\s_k = |\I_{\S_k}| = \sum_{\ell=0}^{\lfloor k/2\rfloor} |\I_{\S_k}^{k-2 \ell}|$ denote the total number of involutions in $\S_k$, then $\s_k$ is the degree of $\varphi_{\S_k}$ and is the sum of the dimensions of the irreducible $\S_k$ modules.  The first few values of $\s_k$ are 1, 1, 2, 4, 10, 26, 76, 232, 764, 2620, 9496.
The sequence $\mathsf{s}_k$ is \cite{OEIS}  \href{http://oeis.org/A000085}{A000085} and has the well-known exponential generating function
$
e^{x^2/2+x} = \sum_{k=0}^\infty \mathsf{s}_k \frac{x^k}{k!}.
$

\end{subsection}

\begin{subsection}{The model representation of $\S_k$}

We  now construct the corresponding induced model representation. For $t \in \I_{\S_k}^{f}$,  let $\W_t = \CC t$ be the 1-dimensional $\C(t)$-module with character $\pi_f$, so that $c \in \C(t)$ acts on $t$ by $c \cdot t = S(c,t)\, c t c^{-1} = S(c,t) t$, where $S(c,t) = \pi_f(c)$ is the sign of the permutation induced by $c$ on the fixed points of $t$. The cosets of $\C(t)$ in $\S_k$  are in bijection with the  $\S_k$-orbits of $t$, which is the set of involutions $\I_{\S_k}^{f}$.  
Consider the induced module
\begin{equation}
 \Ind_{\C(t)}^{\S_k} ( \W_t) \cong \CC \S_k \otimes_{\C(t)} t,
\end{equation}
where the $\S_k$ action is given by 
$
x \cdot (w \otimes_{\C(t)} t) =  x w \otimes_{\C(t)} t,$ for all $w, x \in \S_k$,
and is extended linearly to $\CC \S_k  \otimes_{\C(t)} \W_t$.   Since $\W_t$ is 1-dimensional, $\dim( \Ind_{\C(t)}^{\S_k} ( \W_t) ) = |\S_k|/|\C(t)| = | \I_{\S_k}^{f}|.$   Let $\{w_s \vert s \in \I_{\S_k}^{f}\}$ be a set of distinct coset representatives of $\C(t) \in \S_k$ such that $w_s t w_s^{-1} = s$.     If $w \in \S_k$ with $w = w_s c \in w_s \C(t)$, then since the tensor product is over $\C(t)$, we have
$
w \otimes_{\C(t)} t = w_sc  \otimes_{\C(t)}   t =  w_s  \otimes_{\C(t)} c \cdot t = \sgn(c,t) w_s  \otimes_{\C(t)} t.
$
Thus, the vectors of the form $w_s \otimes_{\C(t)} t$ span $\Ind_{\C(t)}^{\S_k} ( \W_t)$ and  by comparing dimensions   $\{ w_s  \otimes_{\C(t)} t\  \vert\  s \in \I_{\S_k}^f \}$  is a basis of $ \Ind_{\C(t)}^{\S_k} ( \W_t)$.

The induced module $\Ind_{\C(t)}^{\S_k} ( \W_t)$ has a nice combinatorial construction.  If $w \in \S_k$ and $t \in \I_{\S_k}^f$ then $w t w^{-1} \in \I_{\S_k}^{f}$ is an involution with the same number $f$ of fixed points as $t$.  However, the relative position of the fixed points are permuted in the map $t \mapsto w t w^{-1}$. Define $\sgn(w,t)$ to be the sign of the permutation induced on the fixed points of $t$ under conjugation.  That is,
\begin{equation}\label{SaxlSign}
S(w,t) = (-1)^{ | \{\, 1 \le i < j \le k \, \vert\, t(i) = i,\, t(j) = j, \hbox{ and } w(i) > w(j)\, \}|}.
\end{equation}
For example, when the following involution is conjugated,
$$
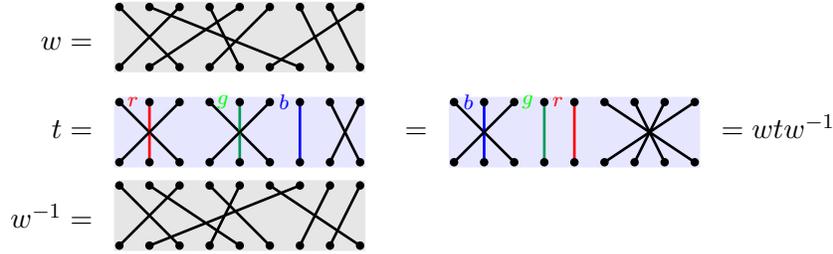
\begin{array}{rc}
w = \begin{array}{c} \begin{tikzpicture}[scale=.4,line width=1pt] 
\foreach \i in {1,...,9} 
{ \path (\i,1) coordinate (T\i); \path (\i,-1) coordinate (B\i); } 
\filldraw[fill= black!10,draw=black!10,line width=4pt]  (T1) -- (T9) -- (B9) -- (B1) -- (T1);
\draw[black] (T1) -- (B3);
\draw[black] (T2) -- (B7);
\draw[black] (T3) -- (B1);
\draw[black] (T4) -- (B5);
\draw[black] (T5) -- (B2);
\draw[black] (T6) -- (B4);
\draw[black] (T7) -- (B8);
\draw[black] (T8) -- (B9);
\draw[black] (T9) -- (B6);
\foreach \i in {1,...,9} 
{ \fill (T\i) circle (4pt); \fill (B\i) circle (4pt); } 
\end{tikzpicture} \end{array} & \\
t = \begin{array}{c}\begin{tikzpicture}[scale=.4,line width=1pt] 
\foreach \i in {1,...,9} 
{ \path (\i,1) coordinate (T\i); \path (\i,-1) coordinate (B\i); } 
\filldraw[fill= blue!10,draw=blue!10,line width=4pt]  (T1) -- (T9) -- (B9) -- (B1) -- (T1);
\draw[black] (T1) -- (B3);
\draw[red] (T2) -- (B2);
\draw[black] (T3) -- (B1);
\draw[black] (T4) -- (B6);
\draw[ForestGreen] (T5) -- (B5);
\draw[black] (T6) -- (B4);
\draw[blue] (T7) -- (B7);
\draw[black] (T8) -- (B9);
\draw[black] (T9) -- (B8);
\draw  (T2)  node[left,red]{$\scriptstyle{r}$};
\draw  (T5)  node[left,green]{$\scriptstyle{g}$};
\draw  (T7)  node[left,blue]{$\scriptstyle{b}$};
\foreach \i in {1,...,9} 
{ \fill (T\i) circle (4pt); \fill (B\i) circle (4pt); } 
\end{tikzpicture} \end{array}&=\begin{array}{c}
\begin{tikzpicture}[scale=.4,line width=1pt] 
\foreach \i in {1,...,9} 
{ \path (\i,1) coordinate (T\i); \path (\i,-1) coordinate (B\i); } 
\filldraw[fill= blue!10,draw=blue!10,line width=4pt]  (T1) -- (T9) -- (B9) -- (B1) -- (T1);
\draw[black] (T1) -- (B3);
\draw[blue] (T2) -- (B2);
\draw[black] (T3) -- (B1);
\draw[ForestGreen] (T4) -- (B4);
\draw[red] (T5) -- (B5);
\draw[black] (T6) -- (B9);
\draw[black] (T7) -- (B8);
\draw[black] (T8) -- (B7);
\draw[black] (T9) -- (B6);
\draw  (T5)  node[left,red]{$\scriptstyle{r}$};
\draw  (T4)  node[left,green]{$\scriptstyle{g}$};
\draw  (T2)  node[left,blue]{$\scriptstyle{b}$};
\foreach \i in {1,...,9} 
{ \fill (T\i) circle (4pt); \fill (B\i) circle (4pt); } 
\end{tikzpicture} \end{array}= wtw^{-1} \\
w^{-1} = \begin{array}{c}
\begin{tikzpicture}[scale=.4,line width=1pt] 
\foreach \i in {1,...,9} 
{ \path (\i,1) coordinate (T\i); \path (\i,-1) coordinate (B\i); } 
\filldraw[fill= black!10,draw=black!10,line width=4pt]  (T1) -- (T9) -- (B9) -- (B1) -- (T1);
\draw[black] (T1) -- (B3);
\draw[black] (T2) -- (B5);
\draw[black] (T3) -- (B1);
\draw[black] (T4) -- (B6);
\draw[black] (T5) -- (B4);
\draw[black] (T6) -- (B9);
\draw[black] (T7) -- (B2);
\draw[black] (T8) -- (B7);
\draw[black] (T9) -- (B8);
\foreach \i in {1,...,9} 
{ \fill (T\i) circle (4pt); \fill (B\i) circle (4pt); } 
\end{tikzpicture} \end{array} & 
\end{array}
$$
the three fixed points $(r,g,b)$ are permuted to $(b,g,r)$ which is an induced permutation of sign $-1$.

Inside of the group algebra $\CC \S_k$ define
\begin{equation}
\M_{\S_k}^f= \CC\text{-span}\left\{\ t \ \vert \ t \in \I_{\S_k}^{f} \right\},
\end{equation}
Define an action of $w \in \S_k$ on a basis element $t \in \I_{\S_k}^{f}$ by 
\begin{equation}\label{SymmetricGroupSignedConjugation}
w \cdot t = \sgn(w,t) w t w^{-1},
\end{equation}
which we refer to as \emph{signed conjugation}.  
Let $\S_k$ act on $\M_{\S_k}^f$ by extending the action of \eqref{SymmetricGroupSignedConjugation} linearly.

\begin{prop}\label{prop:Model}  For  $f =k - 2 \ell$ with $0 \le \ell \le  \lfloor k/2 \rfloor$,  we have $\M_{\S_k}^f \cong \Ind_{\C(t)}^{\S_k} ( \W_t)$.

\end{prop}

\begin{proof} Let $s_i = (i,i+1)$ denote the simple transposition (given here in cycle notation) that exchanges $i$ and $i+1$. Then $s_1, \ldots, s_{k-1}$ generate $\S_k$, and the length $\ell(w)$ of $w \in \S_k$ is the minimum number of simple transpositions needed to write $w$ as a product of simple transpositions. Consider the coset $w \C(t)$ in $\S_k$ and let $w t w^{-1} =s$.
We claim that if $w$ is of minimal length among all permutations in $w \C(t)$, then under the map $t \mapsto w t w^{-1}  = s$ the relative position of the the fixed points of $t$ is not changed. This can be readily verified from the diagram calculus: the length $\ell(w)$ is the number of crossings in the permutation diagram of $w$, and thus the permutation of minimal length that conjugates $t$ to $s$ does not exchange any of the fixed points of $t$.  Now, for $s \in \I_{\S_k}^{f}$ let $w_s$ be the unique minimal-length coset representative such that $w_s t w_s^{-1}  = s$.  Then 
 $\{\ w_s\  \vert\  s \in \I_{\S_k}^{f}\ \}$ is a set of distinct  coset representatives of $\C(t) \in \S_k$.   

We now show that $x \in \S_k$ acts on the basis $\{\ w_s \otimes_{\C(t)} t \  \vert\  s \in \I_{\S_k}^{f}\ \}$ of $\Ind_{\C(t)}^{\S_k} ( \W_t)$ identically to the way that $x$ acts on the basis $\{\, s \in \I_{\S_k}^{f}\, \}$ of $\M_{\S_k}^f$.  We know that  $x w_s \in w_r \C(t)$ for some $r \in \I_{\S_k}^{f}$ so $x w_s  =  w_r c$ for $c \in \C(t)$, and thus
$$
x \cdot (w_s \otimes_{\C(t)} t) = x w_s \otimes_{\C(t)} t = w_r c  \otimes_{\C(t)} t = w_r  \otimes_{\C(t)} c \cdot t = 
S(c,t) (w_r  \otimes_{\C(t)} t).
$$
Now observe that $x = w_r c w_s^{-1}$ so 
$$
x s x^{-1} = (w_r c w_s^{-1}) (w_s t w_s^{-1}) (w_s c^{-1} w^{-1}_r )
= w_r (c t  c^{-1}) w^{-1}_r = w_r t  w^{-1}_r = r.
$$
Furthermore, since $w_r$ does not change the relative order of the fixed points of $t$, the only place where the relative order of the fixed points of $t$ was changed was in the conjugation $c t  c^{-1} = t$.  Thus $S(x,s) = S(c,t)$ and so $x \cdot s = \sgn(x,s) x s x^{-1}$.
\end{proof}

Now, let $\CC$ be chosen so that $\CC \S_k$ is semisimple (for example, $char(\CC) = 0$), and define
\begin{equation}\label{def:M}
\M_{\S_k} =  \CC\text{-span}\left\{\ t \ \vert \ t\in \I_k  \right\} = \bigoplus_{\ell = 0}^{\lfloor n/2 \rfloor} \M_{\S_k}^{k-2\ell}.
\end{equation}
Let $\Lambda_{\S_k}^f = \{ \lambda \vdash k \vert \odd(\lambda) = f \}$. Applying Proposition \ref{prop:Model}  to  \eqref{def:M} and using \eqref{SaxlThm} gives

\begin{thm} \label{thm:Model} Under signed conjugation \eqref{SignedConjugation},  $\M_{\S_k}$ decomposes into irreducible $\S_k$ modules as
$$
\M_{\S_k}= \bigoplus_{\ell = 0}^{\lfloor n/2 \rfloor} \M_{\S_k}^{k-2\ell} \cong  \bigoplus_{\ell = 0}^{\lfloor n/2 \rfloor} \bigoplus_{\lambda \in \Lambda_{\S_k}^{k-2\ell}} \S_k^\lambda =  \bigoplus_{f = 0}^{k} \bigoplus_{\lambda \in \Lambda_{\S_k}^{f}} \S_k^\lambda =  \bigoplus_{\lambda \vdash n} \S_k^\lambda,
$$
where $\Lambda_{\S_k}^{f} = \emptyset$ if there are no partitions of $k$ with $f$ odd parts.
\end{thm}

\end{subsection}

\begin{subsection}{Comparison with other Gelfand models for $\S_k$}\label{ModelComparisonSection}

Adin, Postnikov, and Roichman \cite{APR} (and also \cite{KM}) study a slightly different combinatorial model for $\S_k$, and it is the analog of this model that Mazorchuk derives for the diagram algebras in \cite{Mz}. The sign in \cite{APR} is computed on the two cycles of $t \in \I_k^f$ as follows:
\begin{equation}
s(w,t) = (-1)^{ | \{\, 1 \le i < j \le k \, \vert\, t(i) = j,\, t(j) = i, \hbox{ and } w(i) > w(j)\, \}|}.
\end{equation}
and the action of $\S_k$ on $\I_k^f$ is given as
$$
w \cdot t = s(w,t) w t w^{-1}, \qquad \hbox{for  $w \in S_k$ and $t \in \I_k^f$}.
$$
We let $\overline{\W}_k^f$ denote the corresponding $\S_k$ module, and let $\overline{\W}_k = \oplus_{\ell = 0}^{\lfloor k/2 \rfloor}  \overline{\W}_k^{k-2\ell}$.
In \cite{APR}  it is proved that $\overline{\W}_k$ is a Gelfand model for $\S_k$.  Moreover, the action is given a $q$-extension in \cite{APR} to a Gelfand model for the Iwahori-Hecke algebra $\H_k(q)$ of $\S_k$.   In what follows we prove that the Adin-Postnikov-Roichman model differs from the Saxl model by the sign representation.

\begin{prop} For each $0 \le f \le k$ such that $k -f$ is even, we have $\W_{\S_k}^f \cong \overline{\W}_{\S_k}^f \otimes \S_k^{\sign}$, where $\S_k^{\sign}$ is the sign representation of $\S_k$.
\end{prop}

\begin{proof}
Let $t \in \I_{\S_k}^f$ and let $w \in \S_k$ such that $w t w^{-1} = t$.  That is, the $t$-$t$ entry of the matrix of $w$ is nonzero (in both $\overline{\W}_{\S_k}^f$ and ${\W}_{\S_k}^f$) and thus contributes to the trace.
Let $F$ be the set of fixed points of $t$ and  let $t = (a_1,b_1) (a_2, b_2) \cdots (a_\ell, b_\ell)$ be the decomposition of $t$ into $\ell = (k-f)/2$ disjoint 2-cycles with $a_i < b_i$ for each $i$.  Thus the complement of $F$ in $\{1,2, \ldots, k\}$ is $\bar F = \{a_1, b_1, a_2, b_2, \ldots, a_\ell, b_\ell\}$.

Since $w t w^{-1} = t$, we know that $w$ permutes the fixed points $F$ of $t$.  Furthermore, $w$ permutes the transpositions among themselves and possibly transposes the endpoints of the transpositions. We factor $w$ according to these three actions.
Let  $w_a, w_b, w_{\pi} \in \S_k$ be defined as follows:
\begin{enumerate}
\item $w_b(c) = c$ if $c \in \bar F$ and $w_b(d) = w(d)$ if $d \in F$;  thus, $w_b$ permutes the fixed points of $t$ as in $w$ and fixes the others.
\item $w_a(d) = d$ if $d \in F$, $w_a(a_i) = b_i$ and $w_a(b_i) = a_i$ if $w(a_i) > w(b_i)$, and $w_a(a_i) = a_i$ and $w_a(b_i) = b_i$ if $w(a_i) < w(b_i)$; thus $w_a$ transposes the endpoints of the transpositions in $t$ if they are transposed by $w$.
\item $w_\pi(d) = d$ if $d \in F$ and $w_\pi(a_i) = a_{\pi(i)}$ and $w_\pi(b_i) = b_{\pi(i)}$ where $\pi$ is the permutation on the transpositions induced by $w$.
\end{enumerate}
It is easy to check, by  examining the values of these permutations on each element of $F \cup \bar F = \{1, \ldots, k\}$, that 
$$
w = w_{\pi} w_a w_b, \quad\hbox{ and thus }\quad \sign(w) = \sign(w_{\pi}) \sign(w_a) \sign(w_b).
$$
Furthermore, by the  definition of $w_a, w_b, S(w,t),$ and $s(w,t)$ we have $\sign(w_b) = S(w,t)$ and  $\sign(w_a) = s(w,t)$.  Finally, since $w_\pi$ permutes the set of transpositions $(a_i,b_i)$, keeping $a_i < b_i$, it can be written as a product of pairs of transpositions of the form $(a_i,a_j)(b_i,b_j)$. Thus, $\sign(w_\pi) = 1$, and we have $\sign(w) = S(w,t) s(w,t)$ or, equivalently,  $S(w,t) = s(w,t) \sign(w)$,  for each $t$ such that $wtw^{-1} = t$.  It follows that the characters of $w$ on $\M_{\S_k}^f$ and $\overline{\W}_{\S_k}^f \otimes \S_k^{\sign}$ are equal and  the modules are isomorphic. \end{proof}

\end{subsection}

\end{section}

\begin{section}{Gelfand Models from  the Jones Basic Construction}

In this section we show how to construct model representations for a tower of algebras $(\A_k)_{k \ge 0}$ that is obtained by repeated Jones basic constructions (\cite{GHJ}, \cite{Jo1}, \cite{GG}) from a tower of algebras $(\C_k)_{k \ge0}$ for which we already have a model representation. Each of the diagram algebras of interest in this paper has such an inductive structure with $\C_k \cong \CC \S_k$, in the case of the nonplanar algebras, and $\C_k \cong \CC {\bf 1}$, in the case of the planar algebras.  We lift Saxl's model for $\S_k$ in the first case and the trivial model in the second.  

\begin{subsection}{The Jones basic construction}
\label{sec:JBC}

Let $\A_0 \subseteq \A_1 \subseteq \A_2 \subseteq \cdots $ be a tower of inclusions of finite-dimensional, semisimple algebras with 1 over the algebraically closed field $\CC$.   Assume that  there exist elements $0 \not= \e_k \in \A_{k+1}$, and $k' < k$, which satisfy the following relations \begin{enumerate}
\item[(a)] $\e_k^2 = \e_k$,
\item[(b)] $\e_k a = a \e_k$, for all $a \in \A_{k'}$,
\item[(c)] $\A_{k'} \cong \A_{k'} \e_k \cong \e_k \A_k \e_k$ via the map $a \mapsto a \e_k$ for all $a \in \A_k$.
\end{enumerate}
In the examples in this paper  $k' = k-1$ or $k' = k-2$. Define the map $\varepsilon_k: \A_k \to \A_{k'}$, called the conditional expectation, such that $\varepsilon_k(b)$ is the unique element in $\A_{k'}$ such that $\e_k b \e_k = \varepsilon_k(b) \e_k$. Let $A_k \e_k$ be a module for $\A_k \e_k \A_k$  by multiplication on the left and a module for $\e_k \A_k \e_k$ by multiplication on the right. Wenzl  \cite{Wz} proves the following (see also \cite[Theorem 2.6]{HR1}, \cite[Prop. 5.1.3]{Ha1}),  
\begin{equation}\label{SWDuality}
\J_k := \A_k \e_k \A_k \cong \End_{\e_k \A_k \e_k}(\A_k \e_k) \qquad\hbox{and}\qquad \A_{k'} \cong \e_k \A_k \e_k \cong \End_{ \A_k \e_k \A_k}(\A_k \e_k).
\end{equation}

Since $\J_k = \A_k \e_k \A_k \cong  \End_{\e_k \A_k \e_k}(\A_k \e_k)$ it is an ideal and a semisimple subalgebra (with unit $\e_k$).  Thus, there exists a subalgebra $\C_k \subseteq \A_k$ such that
\begin{equation} \label{ideal}
\A_k = \J_k \oplus \C_k \qquad\hbox{and}\qquad \A_k/\J_k \cong \C_k.
\end{equation}
Let $\Lambda_{\A_k}$ index the irreducible $\A_k$-modules. It follows (from double centralizer theory) that the irreducible modules for $\A_{k'}$ and $\J_k$ are indexed by the same set $\Lambda_{\A_{k'}}$.  
Thus by \eqref{ideal}, we have
\begin{equation}\label{inducctivereplabels}
\Lambda_{\A_k} = \Lambda_{\J_{k}} \sqcup \Lambda_{\C_k} = \Lambda_{\A_{k'}} \sqcup \Lambda_{\C_k}.
\end{equation}
Applying \eqref{inducctivereplabels} recursively gives
$
\Lambda_{\A_k} = \Lambda_{\C_0} \sqcup \Lambda_{\C_1} \sqcup \Lambda_{\C_2} \sqcup \cdots \sqcup \Lambda_{\C_k},
$
where for some values of $i$ we may have $\C_i = 0$ and $\Lambda_{\C_i} = \emptyset$ (see the examples in  \ref{ModelDiags}).  

If $\chi$ is any character of $\A_k$, then by \eqref{ideal}, $\chi$ is completely determined by its values on $\J_{k}$ and $\C_k$.  If $a \in \J_{k} = \A_k \e_k \A_k$, then $a = a_1 \e_k a_2$ for $a_1, a_2 \in \A_k,$ and by the trace property of $\chi$,
$$
\chi(a) = \chi(a_1 \e_k a_2)  = \chi(a_2 a_1 \e_k) =  \chi(a_2 a_1 \e_k^2) =   \chi(\e_k a_2 a_1 \e_k)
=   \chi(\varepsilon(a_2 a_1) \e_k).
$$
Thus, character values on $\J_k$ are completely determined by their values on $a \e_k$ for $a \in \A_{k'}$, and
\begin{equation}\label{CharacterSufficiency}
\begin{array}{l}
\text{\sl Characters of $\A_k$ are completely determined by their values on}\\
\text{\sl $b \in \C_k$ and $a \e_k$, for $a \in \A_{k'}$.}\\
\end{array}
\end{equation}

The following result is proved in \cite{HR1} and \cite{Ha1} for  algebras $\A_k$ satisfying (a), (b), (c) above  with quotient $\C_k$ defined as in \eqref{ideal}. If $\lambda \in \Lambda_{\A_k}$, then
\begin{equation}\label{JBCIrreducibleBlockDiagonal}
\chi_{\A_k}^\lambda(a) = 
\begin{cases}
\chi_{\C_k}^\lambda(a), & \text{if $\lambda \in \Lambda_{\C_k}$ and $a \in \C_k$,} \\
0, & \text{if $\lambda \in \Lambda_{\C_k}$ and  $a \in \J_k$,} \\
 \chi_{\A_{k'}}^\lambda(a'), & \text{if $\lambda \in \Lambda_{\A_{k'}}$ and $a = a' \e_k$ with $a' \in \A_{k'}$.} \\
\end{cases}
\end{equation}
The  character values, $\chi^\lambda_{\A_k}(a)$ for $\lambda \in \Lambda_{\A_{k'}}$ and $a \in \C_k$, are harder to compute but   not needed here.

Now  assume that we have a model representation $\M_{\C_r}$ of $\C_r$, for each $0 \le r \le  k$, with corresponding character $\varphi_{\C_k}$. Thus,
\begin{equation}\label{CModel}
\M_{\C_r} \cong \bigoplus_{\lambda \in \Lambda_{\C_r}} \C^\lambda_r\qquad\hbox{and}\qquad
\varphi_{\C_r} = \sum_{\lambda \in \Lambda_{\C_r}} \chi^\lambda_{\C_r},
\end{equation}
where $\{\C_r^\lambda\ \vert\ \lambda \in \Lambda_{\C_r}\}$ is the set of the irreducible $\C_r$-modules with characters $\chi^\lambda_{\C_r}, \lambda \in \Lambda_{\C_r}$.
Suppose further that we have constructed a module $\M_{\A_k}$  for $\A_k$ which decomposes into $\A_k$-submodules $\M_{\A_k} = \bigoplus_{r = 0}^k \M_{\A_k}^r$ satisfying

\begin{enumerate}

\item[(M1)] $\M_{\A_k}^r$ and $\M_{\A_k}^s$ have no common irreducible constituents if $r \not= s$, and

\item[(M2)] The character  $\varphi_{\A_k}^r$ of $\M_{\A_k}^r$ satisfies
\begin{equation}\label{JBCModelBlockDiagonal}
\varphi_{\A_k}^r(a) = 
\begin{cases}
\varphi_{\C_k}(a), & \text{if $r=k$ and $a \in \C_k$,} \\
0, & \text{if $r=k$ and  $a \in \J_k$,} \\
 \varphi_{\A_{k}'}^r(a'), & \text{if $r<k$ and $a = a' \e_k$ with $a' \in \A_{k-1}$.} \\
\end{cases}
\end{equation}
\end{enumerate}
Then the following theorem tells us that $\M_{\A_k}$ is a model for $\A_k$.

\begin{thm} \label{JBCModelCharacters} Let $\A_0 \subseteq \A_1 \subseteq \cdots \subseteq \A_k$ be a tower of semisimple algebras obtained by Jones basic constructions 
from $\C_0 \subseteq \C_1 \subseteq \cdots \subseteq \C_k$.  If $\M_k =  \bigoplus_{r = 0}^k \M_{\A_k}^r$ is an $\A_k$ module satisfying (M1) and (M2), then
 $$
\displaystyle{\M_{\A_k}^{r} = \bigoplus_{\lambda \in \Lambda_{\C_r}}  \A_k^\lambda,}
\qquad\hbox{and thus}\qquad
\displaystyle{\M_{\A_k} = \bigoplus_{r=0}^k  \M_{\A_k}^{r} =  \bigoplus_{\lambda \in \Lambda_{\A_k}}  \M_{\A_k}^\lambda.}
$$
\end{thm}

\begin{proof} We prove this on the character level; namely, we show that $\varphi_{\A_k}^{r} = \sum_{\lambda \in \Lambda_{\C_r}}  \chi_{\A_k}^\lambda$ and 
$\varphi_{\A_k} = \sum_{r=0}^k  \varphi_{\A_k}^{r} =  \sum_{\lambda \in \Lambda_{\A_k}}  \chi_{\A_k}^\lambda$
(the second statement follows immediately from the first).
Our proof is by induction on $k$, with $k=0$ being trivial.  Let $k > 0$ and first consider the situation where $r = k$.
Using \eqref{JBCModelBlockDiagonal},  \eqref{CModel}, and \eqref{JBCIrreducibleBlockDiagonal} we have
\begin{align*}
\varphi_{\A_k}^{k}(a) &= \varphi_{\C_k}(a) = \sum_{\lambda \in \Lambda_{\C_k}} \chi_{\C_k}^\lambda(a) = \sum_{\lambda \in \Lambda_{\C_k}} \chi_{\A_k}^\lambda(a),
\qquad\hbox{for all $a \in \C_k$,\ and } \\
\varphi_{\A_k}^{k}(a) & =0= \sum_{\lambda \in \Lambda_{\C_k}} \chi_{\A_k}^\lambda(a), \hskip1.7in \hbox{for all $a \in \J_k$.}
\end{align*}
Since  characters of $\A_k$ are completely determined by their values on $\C_k$ and $\J_k$, we have that $\varphi_{\A_k}^{k} = \sum_{\lambda \in \Lambda_{\C_k}} \chi_{\A_k}^\lambda$ is the decomposition of $\varphi_{\A_k}^{k}$ into irreducible characters.

Now let $r < k$.  The previous paragraph and   (M1) tell us that the decomposition of $\varphi_{\A_k}^{r}$ into irreducibles does not involve 
any  $\chi^\lambda_{\A_k}$ with $\lambda \in \Lambda_{\C_k}$.  Thus by \eqref{inducctivereplabels} it is of the form  
\begin{equation}\label{inductivesum}
\varphi_{\A_k}^r =    \sum_{\lambda \in \Lambda_{\A_{k'}}}   a_\lambda  \chi_{\A_{k}}^\lambda,
\qquad\hbox{ for some $a_\lambda \in \ZZ_{\ge0}$.}
\end{equation}
Let $a = a'  \e_k$ with $a' \in \A_{k'}$. For each $\lambda \in \Lambda_{\A_{k'}}$ we have $\chi_{\A_{k}}^\lambda(a) = \chi_{\A_{k'}}^\lambda(a')$ by \eqref{JBCIrreducibleBlockDiagonal}, and thus $\varphi_{\A_k}^r(a) = \varphi_{\A_{k'}}^r(a')$ by \eqref{inductivesum}.  Furthermore, 
we can apply the inductive hypothesis since $k'< k$, 
$$
\varphi_{\A_k}^{r}(a)  =    \varphi_{\A_{k'}}^{r} (a') 
=   \sum_{\lambda \in \Lambda_{\C_r}}   \chi_{\A_{k'}}^\lambda (a')
=   \sum_{\lambda \in \Lambda_{\C_r}}   \chi_{\A_{k}}^\lambda (a).
$$
Thus,  $a_\lambda = 1$ for $\lambda \in \Lambda_{\C_r}$ and $a_\lambda = 0$, otherwise, as desired.\end{proof}

\end{subsection}

\begin{subsection}{Jones basic construction for subalgebras of the partition algebra}

Let $\A_k$ be the partition algebra or one of its subalgebras described in  \ref{sec:subalgebras} with   $x \in \CC$ chosen such that $\A_k$ is semisimple.  Let $\AA_k$ be the basis of diagrams which span $\A_k$.  There is a natural embedding
$\A_{k-1} \subseteq \A_k,$  given by placing an identity edge to the right of any diagram  $d \in \A_{k-1}$, i.e., 
$
\begin{array}{c}
\begin{tikzpicture}[scale=.4,line width=1pt] 
\foreach \i in {1,...,4} 
{ \path (\i,.4) coordinate (T\i); \path (\i,-.4) coordinate (B\i); } 
\filldraw[fill= black!10,draw=black!10,line width=4pt]  (T1) -- (T4) -- (B4) -- (B1) -- (T1) -- (T2);
\draw (2.5,0) node {$d$};
\end{tikzpicture}\end{array}
\mapsto
\begin{array}{c}
\begin{tikzpicture}[scale=.4,line width=1pt] 
\foreach \i in {1,...,5} 
{ \path (\i,.4) coordinate (T\i); \path (\i,-.4) coordinate (B\i); } 
\filldraw[fill= black!10,draw=black!10,line width=4pt]  (T1) -- (T4) -- (B4) -- (B1) -- (T1) -- (T2);
\draw[black] (T5) -- (B5);
\foreach \i in {5} 
{ \fill (T\i) circle (4pt); \fill (B\i) circle (4pt); } 
\draw (2.5,0) node {$d$};
\end{tikzpicture}\end{array}.
$ Let $\J_{k} \subseteq \A_k$ be the ideal spanned by the diagrams of $\A_k$ having rank less than $k$.  Then,
\begin{equation}\label{BasicConstruction}
\A_k \cong \J_{k} \oplus \C_k,
\end{equation}
where $\C_k$ is the span of the diagrams of rank exactly equal to $k$.  In the examples of this paper,
\begin{equation}
\begin{array}{lll}
\C_k \cong \CC \S_k, & & \hbox{when $\A_k$ is one of the nonplanar algebras  $\P_k(x), \B_k(x), \RB_k(x)$ or  $\R_k$}, \\
\C_k \cong \CC{\bf 1}_k, && \hbox{when $\A_k$  is one of the planar algebras $\TL_k(x),  \Motz_k(x),$ or $\PR_k$.} \\
\end{array}
\end{equation}

Define an idempotent $\e_k \in \J_{k}$ by
\begin{equation}\label{essential}
\begin{array}{rclll}
 \e_k &=& \frac{1}{x}
\begin{array}{c}
\begin{tikzpicture}[scale=.5,line width=1pt] 
\foreach \i in {1,...,6} 
{ \path (\i,.5) coordinate (T\i); \path (\i,-.5) coordinate (B\i); } 
\filldraw[fill= black!10,draw=black!10,line width=4pt]  (T1) -- (T6) -- (B6) -- (B1) -- (T1);
\draw[black] (T1) -- (B1);
\draw[black] (T2) -- (B2);
\draw[black] (T4) -- (B4);
\draw[black] (T5) -- (B5);
\foreach \i in {1,2,4,5,6} 
{ \fill (T\i) circle (4pt); \fill (B\i) circle (4pt); }
\draw  (T1)  node[above=0.1cm]{$\scriptstyle{1}$};
\draw  (T2)  node[above=0.1cm]{$\scriptstyle{2}$};
\draw  (T3)  node[]{$\cdots$};\draw  (B3)  node[]{$\cdots$};
\draw  (T5)  node[above=0.1cm]{$\scriptstyle{k-1}$};
\draw  (T6)  node[above=0.1cm]{$\scriptstyle{k}$};
\end{tikzpicture} 
\end{array}
, & & \text{for $\A_k$ equal to $\P_k(x), \RB_k(x), \R_k, \Motz_k(x),$ or  $\PR_k$,} \\
\e_k &=& \frac{1}{x} \begin{array}{c}
\begin{tikzpicture}[scale=.5,line width=1pt] 
\foreach \i in {1,...,6} 
{ \path (\i,.5) coordinate (T\i); \path (\i,-.5) coordinate (B\i); } 
\filldraw[fill= black!10,draw=black!10,line width=4pt]  (T1) -- (T6) -- (B6) -- (B1) -- (T1);
\draw[black] (T1) -- (B1);
\draw[black] (T2) -- (B2);
\draw[black] (T4) -- (B4);
\draw[black] (T5) ..controls +(.15,-.5) and +(-.15,-.5) .. (T6);
\draw[black] (B5) ..controls +(.15,.5) and +(-.15,.5) .. (B6);
\foreach \i in {1,2,4,5,6} 
{ \fill (T\i) circle (4pt); \fill (B\i) circle (4pt); }
\draw  (T1)  node[above=0.1cm]{$\scriptstyle{1}$};
\draw  (T2)  node[above=0.1cm]{$\scriptstyle{2}$};
\draw  (T3)  node[]{$\cdots$};\draw  (B3)  node[]{$\cdots$};
\draw  (T5)  node[above=0.1cm]{$\scriptstyle{k-1}$};
\draw  (T6)  node[above=0.1cm]{$\scriptstyle{k}$};
\end{tikzpicture} 
\end{array}, & & \text{for $\A_k$ equal to $\B_k(x)$ or $\TL_k(x)$.} 
\end{array}
\end{equation}
Recall that in the special cases where $\A_k = \R_k$ or $\A_k=\PR_k$ we have $x = 1$.
It is easy to verify by diagram multiplication that 
\begin{equation}
\begin{array}{rclll}
\J_{k} &= &  \A_k \e_k \A_k,  & & \text{for $\A_k$ equal to $\P_k(x), \RB_k(x), \R_k, \Motz_k(x),$ or  $\PR_k$,} \\
\J_{k-1} = \J_{k} &= & \A_k \e_k \A_k,  & & \text{for $\A_k$ equal to $\B_k(x)$ or $\TL_k(x)$.} 
\end{array}
\end{equation}
Define $k' = k-1$ or $k' = k-2$ so that
\begin{equation}\label{DerivedAlgebra}
\A_{k'}  = \begin{cases}
\A_{k-1}, & \text{if $\A_k$ equals $\P_k(x), \RB_k(x), \R_k, \Motz_k(x),$ or  $\PR_k$,} \\
\A_{k-2}, & \text{if $\A_k$ equals  $\B_k(x)$ or $\TL_k(x)$.}
\end{cases} 
\end{equation}

In each of our examples it is well-known, and straight-forward to verify, that $\A_k$ satisfies  properties (a), (b), (c) of the basic construction in \ref{sec:JBC} using the idempotent $\e_k$. Thus,  $\Lambda_{\A_k} = \Lambda_{\J_{k-1}} \sqcup \Lambda_{\C_k}  = \Lambda_{\A_k'} \sqcup \Lambda_{\C_k}$, and by induction 
 \begin{equation} \label{IrreducibleLabels2}
\begin{array}{rclll}
 \Lambda_{\A_k}  &= & \bigsqcup_{r = 0}^k \Lambda_{\C_r},  & & \text{for $\A_k$ equal to $\P_k(x), \RB_k(x), \R_k, \Motz_k(x),$ or  $\PR_k$,} \\
 \Lambda_{\A_k}  &= & \bigsqcup_{\ell = 0}^{\lfloor k/2 \rfloor} \Lambda_{\C_{k-2 \ell}},   & & \text{for $\A_k$ equal to $\B_k(x)$ or $\TL_k(x)$.} 
\end{array}
\end{equation}
\end{subsection}

\begin{subsection}{Symmetric diagrams and diagram conjugation}\label{SymmetricDiagramSection}

 For any  diagram $d \in \AA_k$,  let $d^T \in \AA_k$  be the diagram obtained by reflecting $d$ over its horizontal axis.  For example, 
$$
\begin{array}{lcl}
d_1=\begin{array}{c}
\begin{tikzpicture}[scale=.4,line width=1pt] 
\foreach \i in {1,...,10} 
{ \path (\i,1) coordinate (T\i); \path (\i,-1) coordinate (B\i); } 
\filldraw[fill= black!10,draw=black!10,line width=4pt]  (T1) -- (T10) -- (B10) -- (B1) -- (T1);
\draw[black] (T1) -- (B3).. controls +(.1,.75) and +(-.1,.75) .. (B4);
\draw[black] (B1)--(T2) ..controls +(.15,-.5) and +(-.15,-.5) .. (T3);
\draw[black] (B1) .. controls +(.1,.5) and +(-.1,.5) .. (B2);
\draw[black] (T4) -- (B6);
\draw[black] (T5) -- (B5);
\draw[black] (T6)..controls +(.1,-.75) and +(-.1,-.75) .. (T7);
\draw[black] (T9) -- (B10);
\draw[black] (T10) -- (B8);
\draw[black]  (B7).. controls +(.1,.75) and +(-.1,.75) .. (B9);
\foreach \i in {1,...,10} 
{ \fill (T\i) circle (4pt); \fill (B\i) circle (4pt); }
\end{tikzpicture}\end{array}
& \Rightarrow &
d_1^T=\begin{array}{c}
\begin{tikzpicture}[scale=.4,line width=1pt] 
\foreach \i in {1,...,10} 
{ \path (\i,1) coordinate (T\i); \path (\i,-1) coordinate (B\i); } 
\filldraw[fill= black!10,draw=black!10,line width=4pt]  (T1) -- (T10) -- (B10) -- (B1) -- (T1);
\draw[black] (B1) -- (T3).. controls +(.1,-.75) and +(-.1,-.75) .. (T4);
\draw[black] (T1)--(B2) ..controls +(.15,.5) and +(-.15,.5) .. (B3);
\draw[black] (T1) .. controls +(.1,-.5) and +(-.1,-.5) .. (T2);
\draw[black] (B4) -- (T6);
\draw[black] (B5) -- (T5);
\draw[black] (B6)..controls +(.1,.75) and +(-.1,.75) .. (B7);
\draw[black] (B9) -- (T10);
\draw[black] (B10) -- (T8);
\draw[black]  (T7).. controls +(.1,-.75) and +(-.1,-.75) .. (T9);
\foreach \i in {1,...,10} 
{ \fill (T\i) circle (4pt); \fill (B\i) circle (4pt); }
\end{tikzpicture}\end{array},  \\
d_2 = \begin{array}{c}
\begin{tikzpicture}[scale=.4,line width=1pt] 
\foreach \i in {1,...,10} 
{ \path (\i,1) coordinate (T\i); \path (\i,-1) coordinate (B\i); } 
\filldraw[fill= black!10,draw=black!10,line width=4pt]  (T1) -- (T10) -- (B10) -- (B1) -- (T1);
\draw[black] (T1) -- (B3);
\draw[black] (T2) -- (B4);
\draw[black] (T3) -- (B1);
\draw[black] (B2)--(T4);
\draw[black] (T6)--(B6);
\draw[black] (T4)  .. controls +(.1,-.75) and +(-.1,-.75) ..  (T5);
\draw[black] (B4) .. controls +(.1,.75) and +(-.1,.75) .. (B5);
\draw[black] (T6) .. controls +(.1,-.5) and +(-.1,-.75) .. (T7) ;
\draw[black] (B6) .. controls +(.1,.5) and +(-.1,.75) .. (B7) ;
\draw[black] (T7)  .. controls +(.1,-.75) and +(-.1,-.75) ..  (T10);
\draw[black] (B7) .. controls +(.1,.75) and +(-.1,.75) .. (B10) ;
\draw[black] (T8) -- (B8);
\foreach \i in {1,...,10} 
{ \fill (T\i) circle (4pt); \fill (B\i) circle (4pt); }
\end{tikzpicture} 
\end{array} 
& \Rightarrow &
d_2^T=\begin{array}{c}
\begin{tikzpicture}[scale=.4,line width=1pt] 
\foreach \i in {1,...,10} 
{ \path (\i,1) coordinate (T\i); \path (\i,-1) coordinate (B\i); } 
\filldraw[fill= black!10,draw=black!10,line width=4pt]  (T1) -- (T10) -- (B10) -- (B1) -- (T1);
\draw[black] (T1) -- (B3);
\draw[black] (T2) -- (B4);
\draw[black] (T3) -- (B1);
\draw[black] (B2)--(T4);
\draw[black] (T6)--(B6);
\draw[black] (T4)  .. controls +(.1,-.75) and +(-.1,-.75) ..  (T5);
\draw[black] (B4) .. controls +(.1,.75) and +(-.1,.75) .. (B5);
\draw[black] (T6) .. controls +(.1,-.5) and +(-.1,-.75) .. (T7) ;
\draw[black] (B6) .. controls +(.1,.5) and +(-.1,.75) .. (B7) ;
\draw[black] (T7)  .. controls +(.1,-.75) and +(-.1,-.75) ..  (T10);
\draw[black] (B7) .. controls +(.1,.75) and +(-.1,.75) .. (B10) ;
\draw[black] (T8) -- (B8);
\foreach \i in {1,...,10} 
{ \fill (T\i) circle (4pt); \fill (B\i) circle (4pt); }
\end{tikzpicture} 
\end{array}.
\end{array}
$$
Note that the map $d \to d^T$ corresponds to exchanging $i \leftrightarrow i'$ for all $i$.

We say that a diagram $d$ is \emph{symmetric} if $d^T = d$.  In our example above, $d_2$ is symmetric and $d_1$ is not.  If we let $(i')' = i$ and let $B' = \{\, b' \, \vert \, b \in B\, \}$ for a block $B$ of a partition diagram $d$, then $d$ is symmetric if it satisfies: $B \in d$ if and only if $B' \in d$.   
 If $d$ is a partition diagram, then we say that a block $B \in d$ is a \emph{fixed block} if $B'=B$.  In our above examples, $d_1$ has one fixed block, $\{5,5'\}$, and $d_2$ has two fixed blocks, $\{8,8'\}$ and $\{6,7,10,6',7',10'\}$.
Note that 
$$
(ab)^T = b^T a^T \qquad\hbox{and}\qquad (d t d^T)^T = (d^T)^T t^T d^T = d t^T d^T,
 $$
so $t$ is symmetric if and only if $d t d^T$ is symmetric.  We say that $d t d^T$ is the \emph{conjugate} of $t$ by $d$.  See Example \ref{ConjugateExample}.  Observe that in $\P_k(x)$ the sizes of the blocks can change  under conjugation.

\begin{rem}\label{TypeB}
The symmetric diagrams in the partition algebra are referred to as type-$B$ set partitions in \cite{OEIS}   \href{http://oeis.org/A002872}{A002872}.  They appear in \cite{Mo} and  are equivalent to the horizontally symmetric $2 \times n$ letter arrays $H_{2,n}$ in \cite{Qu}. They are closely related to the type-$B$ set partitions used in \cite{Re}, except that in  \cite{Re} the partitions are restricted to have  at most one fixed block.
\end{rem}

\begin{rem} If we restrict our diagrams to the symmetric group, then $d^T$ equals $d^{-1}$, diagram conjugation corresponds to usual group conjugation,  symmetric diagrams are involutions, and fixed blocks are fixed points.
\end{rem}

\begin{examp}\label{ConjugateExample}
\begin{enumerate}[(a)]
\item Conjugation in the partition algebra $\P_k(x)$. 
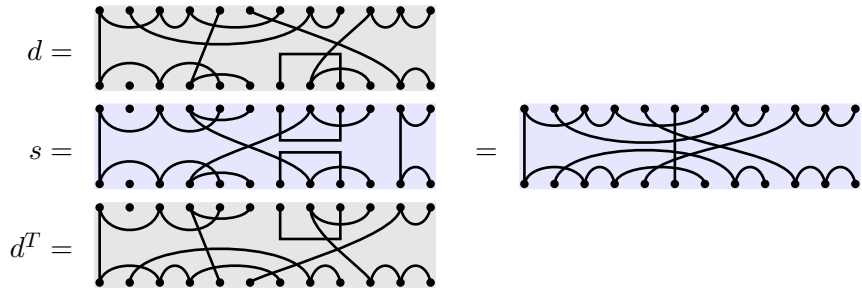
$$
\begin{array}{rc}
 d=\begin{array}{c}\begin{tikzpicture}[scale=.4,line width=1pt] 
\foreach \i in {1,...,12} 
{ \path (\i,1.25) coordinate (T\i); \path (\i,-1.25) coordinate (B\i); } 
\filldraw[fill= black!10,draw=black!10,line width=4pt]  (T1) -- (T12) -- (B12) -- (B1) -- (T1);
\draw[black] (T1) ..controls +(.15,-.75) and +(-.15,-.75) ..(T3);
\draw[black] (T3)  ..controls +(.15,-.75) and +(-.15,-.75) ..(T4);
\draw[black] (T4) ..controls +(.15,-.75) and +(-.15,-.75) ..(T7);
\draw[black] (T2)  ..controls +(.15,-1.5) and +(-.15,-1.5) .. (T8);
\draw[black] (T8) ..controls +(.15,-.75) and +(-.15,-.75) .. (T9);
\draw[black]  (T10)..controls +(.15,-.75) and +(-.15,-.75) ..(T11);
\draw[black]  (T11)..controls +(.15,-.75) and +(-.15,-.75) ..(T12);
\draw[black]( T1)-- (B1);
\draw[black] (T5) -- (B4);
\draw[black] (T6) ..controls +(.15,-.25) and +(0,1) .. (B11);
\draw[black] (T10) ..controls +(.15,-.25) and +(0,1) .. (B8);
\draw[black] (B1) ..controls +(.15,1.0) and +(-.15,1.0) ..(B3)  ;
\draw[black] (B3) ..controls +(.15,1.0) and +(-.15,1.0) ..(B5) ;
\draw[black] (B4)  ..controls +(.15,.5) and +(-.15,.5) .. (B6);
\draw[black] (B8) ..controls +(.15,.75) and +(-.15,.75) .. (B10);
\draw[black]  (B11)..controls +(.15,.75) and +(-.15,.75) ..(B12);
\draw[black] (B7) --(7,-.2)--(9,-.2)-- (B9);
\foreach \i in {1,...,12} 
{ \fill (T\i) circle (4pt); \fill (B\i) circle (4pt); } 
\end{tikzpicture} \end{array}& \\
s= \begin{array}{c}\begin{tikzpicture}[scale=.4,line width=1pt] 
\foreach \i in {1,...,12} 
{ \path (\i,1.25) coordinate (T\i); \path (\i,-1.25) coordinate (B\i); } 
\filldraw[fill= blue!10,draw=blue!10,line width=4pt]  (T1) -- (T12) -- (B12) -- (B1) -- (T1);
\draw[black] (T1) ..controls +(.15,-1.0) and +(-.15,-1.0) ..(T3)  ;
\draw[black] (T3) ..controls +(.15,-1.0) and +(-.15,-1.0) ..(T5) ;
\draw[black] (T4)  ..controls +(.15,-.5) and +(-.15,-.5) .. (T6);
\draw[black] (T8) ..controls +(.15,-.75) and +(-.15,-.75) .. (T10);
\draw[black]  (T11)..controls +(.15,-.75) and +(-.15,-.75) ..(T12);
\draw[black] (T7) --(7,.2)--(9,.2)-- (T9);
\draw[black] (T1)-- (B1);
\draw[black] (T4) ..controls +(0,-1.0) and +(-.15,1.0) .. (B8);
\draw[black] (B4) ..controls +(0,1.0) and +(-.15,-1.0) .. (T8);
\draw[black] (T11) -- (B11);
\draw[black] (B1) ..controls +(.15,1.0) and +(-.15,1.0) ..(B3)  ;
\draw[black] (B3) ..controls +(.15,1.0) and +(-.15,1.0) ..(B5) ;
\draw[black] (B4)  ..controls +(.15,.5) and +(-.15,.5) .. (B6);
\draw[black] (B8) ..controls +(.15,.75) and +(-.15,.75) .. (B10);
\draw[black]  (B11)..controls +(.15,.75) and +(-.15,.75) ..(B12);
\draw[black] (B7) --(7,-.2)--(9,-.2)-- (B9);
\foreach \i in {1,...,12} 
{ \fill (T\i) circle (4pt); \fill (B\i) circle (4pt); } 
\end{tikzpicture} \end{array} &
=\begin{array}{c} \begin{tikzpicture}[scale=.4,line width=1pt] 
\foreach \i in {1,...,12} 
{ \path (\i,1.25) coordinate (T\i); \path (\i,-1.25) coordinate (B\i); } 
\filldraw[fill= blue!10,draw=blue!10,line width=4pt]  (T1) -- (T12) -- (B12) -- (B1) -- (T1);
\draw[black] (T1) ..controls +(.15,-.75) and +(-.15,-.75) ..(T3);
\draw[black] (T3)  ..controls +(.15,-.75) and +(-.15,-.75) ..(T4);
\draw[black] (T4) ..controls +(.15,-.75) and +(-.15,-.75) ..(T7);
\draw[black] (T2)  ..controls +(.15,-1.5) and +(-.15,-1.5) .. (T8);
\draw[black] (T8) ..controls +(.15,-.75) and +(-.15,-.75) .. (T9);
\draw[black]  (T10)..controls +(.15,-.75) and +(-.15,-.75) ..(T11);
\draw[black]  (T11)..controls +(.15,-.75) and +(-.15,-.75) ..(T12);
\draw[black]( T1)-- (B1);
\draw[black]( T6)-- (B6);
\draw[black]( T5)  ..controls +(0,-1.2) and +(0,1.2) ..   (B10);
\draw[black]( T10) ..controls +(0,-1.2) and +(0,1.2) .. (B5);
\draw[black] (B1) ..controls +(.15,.75) and +(-.15,.75) ..(B3);
\draw[black] (B3)  ..controls +(.15,.75) and +(-.15,.75) ..(B4);
\draw[black] (B4) ..controls +(.15,.75) and +(-.15,.75) ..(B7);
\draw[black] (B2)  ..controls +(.15,1.5) and +(-.15,1.5) .. (B8);
\draw[black] (B8) ..controls +(.15,.75) and +(-.15,.75) .. (B9);
\draw[black]  (B10)..controls +(.15,.75) and +(-.15,.75) ..(B11);
\draw[black]  (B11)..controls +(.15,.75) and +(-.15,.75) ..(B12);
\foreach \i in {1,...,12} 
{ \fill (T\i) circle (4pt); \fill (B\i) circle (4pt); } 
\end{tikzpicture} \end{array}\\
d^T = \begin{array}{c}\begin{tikzpicture}[scale=.4,line width=1pt] 
\foreach \i in {1,...,12} 
{ \path (\i,1.25) coordinate (T\i); \path (\i,-1.25) coordinate (B\i); } 
\filldraw[fill= black!10,draw=black!10,line width=4pt]  (T1) -- (T12) -- (B12) -- (B1) -- (T1);
\draw[black] (T1) ..controls +(.15,-1.0) and +(-.15,-1.0) ..(T3)  ;
\draw[black] (T3) ..controls +(.15,-1.0) and +(-.15,-1.0) ..(T5) ;
\draw[black] (T4)  ..controls +(.15,-.5) and +(-.15,-.5) .. (T6);
\draw[black] (T8) ..controls +(.15,-.75) and +(-.15,-.75) .. (T10);
\draw[black]  (T11)..controls +(.15,-.75) and +(-.15,-.75) ..(T12);
\draw[black] (T7) --(7,.2)--(9,.2)-- (T9);
\draw[black]( T1)-- (B1);
\draw[black] (T4) -- (B5);
\draw[black] (T11) ..controls +(0,-1) and +( .15,.25) .. (B6);
\draw[black] (T8) ..controls +(0,-1) and +( .15,.25) .. (B10);
\draw[black] (B1) ..controls +(.15,.75) and +(-.15,.75) ..(B3);
\draw[black] (B3)  ..controls +(.15,.75) and +(-.15,.75) ..(B4);
\draw[black] (B4) ..controls +(.15,.75) and +(-.15,.75) ..(B7);
\draw[black] (B2)  ..controls +(.15,1.5) and +(-.15,1.5) .. (B8);
\draw[black] (B8) ..controls +(.15,.75) and +(-.15,.75) .. (B9);
\draw[black]  (B10)..controls +(.15,.75) and +(-.15,.75) ..(B11);
\draw[black]  (B11)..controls +(.15,.75) and +(-.15,.75) ..(B12);
\foreach \i in {1,...,12} 
{ \fill (T\i) circle (4pt); \fill (B\i) circle (4pt); } 
\end{tikzpicture} \end{array}& \end{array}
$$

\item   Conjugation in the Brauer algebra  $\B_k(x)$:
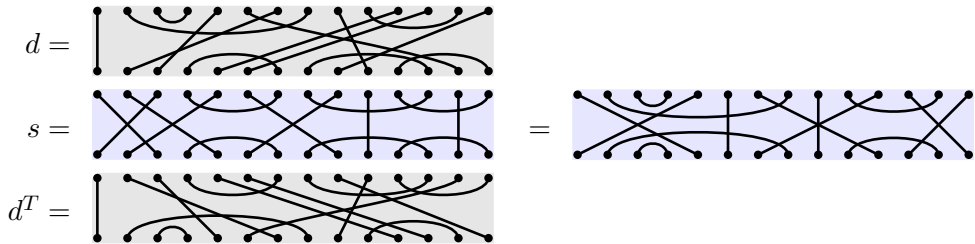
$$
\begin{array}{r}
d =  \begin{array}{c}
\begin{tikzpicture}[scale=.4,line width=1pt] 
\foreach \i in {1,...,14} 
{ \path (\i,1) coordinate (T\i); \path (\i,-1) coordinate (B\i); } 
\filldraw[fill= black!10,draw=black!10,line width=4pt]  (T1) -- (T14) -- (B14) -- (B1) -- (T1);
\draw[black] (T3)  .. controls +(.1,-.5) and +(-.1,-.5) ..  (T4);
\draw[black] (T2)  .. controls +(.1,-1) and +(-.1,-1) ..  (T8);
\draw[black] (T10)  .. controls +(.1,-.75) and +(-.1,-.75) ..  (T13);
\draw[black] (T1) -- (B1);
\draw[black] (T5) -- (B3);
\draw[black] (T6)   .. controls +(0,-.5) and +(0,+.5) ..  (B13);
\draw[black] (T7) -- (B2);
\draw[black] (T11) -- (B5);
\draw[black] (T9) -- (B10);
\draw[black] (T12) -- (B6);
\draw[black] (T14) -- (B9);
\draw[black] (B4)  .. controls +(.1,.75) and +(-.1,.75) ..  (B7);
\draw[black] (B11)  .. controls +(.1,.75) and +(-.1,.75) ..  (B14);
\draw[black] (B8)  .. controls +(.1,.75) and +(-.1,.75) ..  (B12);
\foreach \i in {1,...,14} 
{ \fill (T\i) circle (4pt); \fill (B\i) circle (4pt); }
\end{tikzpicture} 
\end{array}\\
s =  \begin{array}{c} \begin{tikzpicture}[scale=.4,line width=1pt] 
\foreach \i in {1,...,14} 
{ \path (\i,1) coordinate (T\i); \path (\i,-1) coordinate (B\i); } 
\filldraw[fill= blue!10,draw=blue!10,line width=4pt]  (T1) -- (T14) -- (B14) -- (B1) -- (T1);
\draw[black] (T4)  .. controls +(.1,-.75) and +(-.1,-.75) ..  (T7);
\draw[black] (T11)  .. controls +(.1,-.75) and +(-.1,-.75) ..  (T14);
\draw[black] (T8)  .. controls +(.1,-.75) and +(-.1,-.75) ..  (T12);
\draw[black] (T1) -- (B3);
\draw[black] (T3) -- (B1);
\draw[black] (T2) -- (B5);
\draw[black] (B2)--(T5);
\draw[black] (T6)--(B9);
\draw[black] (T9)--(B6);
\draw[black] (T10) -- (B10);
\draw[black] (T13)--(B13);
\draw[black] (B4) .. controls +(.1,.75) and +(-.1,.75) .. (B7);
\draw[black] (B11) .. controls +(.1,.75) and +(-.1,.75) .. (B14);
\draw[black] (B8) .. controls +(.1,.75) and +(-.1,.75) .. (B12) ;
\foreach \i in {1,...,14} 
{ \fill (T\i) circle (4pt); \fill (B\i) circle (4pt); }
\end{tikzpicture}\end{array} \\
d^T =   \begin{array}{c}\begin{tikzpicture}[scale=.4,line width=1pt] 
\foreach \i in {1,...,14} 
{ \path (\i,1) coordinate (T\i); \path (\i,-1) coordinate (B\i); } 
\filldraw[fill= black!10,draw=black!10,line width=4pt]  (T1) -- (T14) -- (B14) -- (B1) -- (T1);
\draw[black] (T4)  .. controls +(.1,-.75) and +(-.1,-.75) ..  (T7);
\draw[black] (T11)  .. controls +(.1,-.75) and +(-.1,-.75) ..  (T14);
\draw[black] (T8)  .. controls +(.1,-.75) and +(-.1,-.75) ..  (T12);
\draw[black] (T1) -- (B1);
\draw[black] (T3) -- (B5);
\draw[black] (T13)   .. controls +(0,-.5) and +(0,.5) ..  (B6);
\draw[black] (T2) -- (B7);
\draw[black] (T5) -- (B11);
\draw[black] (T10) -- (B9);
\draw[black] (T6) -- (B12);
\draw[black] (T9) -- (B14);
\draw[black] (B3)  .. controls +(.1,.5) and +(-.1,.5) ..  (B4);
\draw[black] (B2) .. controls +(.1,1) and +(-.1,1) .. (B8);
\draw[black] (B10) .. controls +(.1,.75) and +(-.1,.75) .. (B13) ;
\foreach \i in {1,...,14} 
{ \fill (T\i) circle (4pt); \fill (B\i) circle (4pt); }
\end{tikzpicture} \end{array}
\end{array} = 
 \begin{array}{c}
\begin{tikzpicture}[scale=.4,line width=1pt] 
\foreach \i in {1,...,14} 
{ \path (\i,1) coordinate (T\i); \path (\i,-1) coordinate (B\i); } 
\filldraw[fill= blue!10,draw=blue!10,line width=4pt]  (T1) -- (T14) -- (B14) -- (B1) -- (T1);
\draw[black] (T3)  .. controls +(.1,-.5) and +(-.1,-.5) ..  (T4);
\draw[black] (T2)  .. controls +(.1,-1) and +(-.1,-1) ..  (T8);
\draw[black] (T10)  .. controls +(.1,-.75) and +(-.1,-.75) ..  (T13);
\draw[black] (T1) -- (B5);
\draw[black] (T5) -- (B1);
\draw[black] (T7) -- (B11);
\draw[black] (T11)--(B7);
\draw[black] (T14)--(B12);
\draw[black] (T12)--(B14);
\draw[black] (T6) -- (B6);
\draw[black] (T9)--(B9);
\draw[black] (B3)  .. controls +(.1,.5) and +(-.1,.5) ..  (B4);
\draw[black] (B2) .. controls +(.1,1) and +(-.1,1) .. (B8);
\draw[black] (B10) .. controls +(.1,.75) and +(-.1,.75) .. (B13) ;
\foreach \i in {1,...,14} 
{ \fill (T\i) circle (4pt); \fill (B\i) circle (4pt); }
\end{tikzpicture} 
\end{array}
$$

\item  Conjugation in the Temperley-Lieb algebra $\TL_k(x)$:
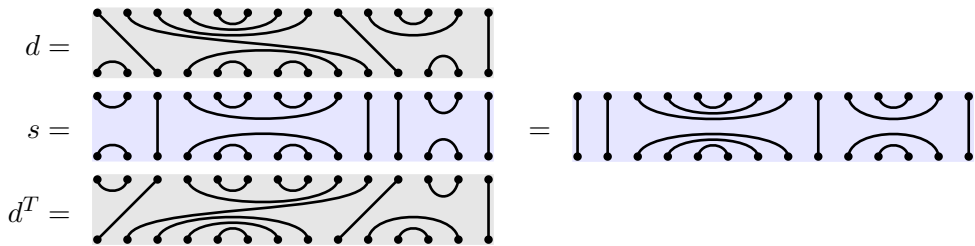
$$
\begin{array}{r}
d =  \begin{array}{c}
\begin{tikzpicture}[scale=.4,line width=1pt] 
\foreach \i in {1,...,14} 
{ \path (\i,1) coordinate (T\i); \path (\i,-1) coordinate (B\i); } 
\filldraw[fill= black!10,draw=black!10,line width=4pt]  (T1) -- (T14) -- (B14) -- (B1) -- (T1);
\draw[black] (T4)  .. controls +(.1,-.75) and +(-.1,-.75) ..  (T7);
\draw[black] (T5)  .. controls +(.1,-.5) and +(-.1,-.5) ..  (T6);
\draw[black] (T3)  .. controls +(.1,-1) and +(-.1,-1) ..  (T8);
\draw[black] (T10)  .. controls +(.1,-1) and +(-.1,-1) ..  (T13);
\draw[black] (T11)  .. controls +(.1,-.5) and +(-.1,-.5) ..  (T12);
\draw[black] (T1) -- (B3);
\draw[black] (T2).. controls +(0,-1.25) and +(0,1.25) .. (B10);
\draw[black] (T9) -- (B11);
\draw[black] (T14) -- (B14);
\draw[black] (B1)  .. controls +(.1,.5) and +(-.1,.5) ..  (B2);
\draw[black] (B4)  .. controls +(.1,1) and +(-.1,1) ..  (B9);
\draw[black] (B5)  .. controls +(.1,.5) and +(-.1,.5) ..  (B6);
\draw[black] (B7)  .. controls +(.1,.5) and +(-.1,.5) ..  (B8);
\draw[black] (B12)  .. controls +(.1,.75) and +(-.1,.75) ..  (B13);
\foreach \i in {1,...,14} 
{ \fill (T\i) circle (4pt); \fill (B\i) circle (4pt); }
\end{tikzpicture} 
\end{array}\\
s =  \begin{array}{c} \begin{tikzpicture}[scale=.4,line width=1pt] 
\foreach \i in {1,...,14} 
{ \path (\i,1) coordinate (T\i); \path (\i,-1) coordinate (B\i); } 
\filldraw[fill= blue!10,draw=blue!10,line width=4pt]  (T1) -- (T14) -- (B14) -- (B1) -- (T1);
\draw[black] (T1)  .. controls +(.1,-.5) and +(-.1,-.5) ..  (T2);
\draw[black] (T4)  .. controls +(.1,-1) and +(-.1,-1) ..  (T9);
\draw[black] (T5)  .. controls +(.1,-.5) and +(-.1,-.5) ..  (T6);
\draw[black] (T7)  .. controls +(.1,-.5) and +(-.1,-.5) ..  (T8);
\draw[black] (T12)  .. controls +(.1,-.75) and +(-.1,-.75) ..  (T13);
\draw[black] (T3) -- (B3);
\draw[black] (T10)--(B10);
\draw[black] (T11)--(B11);
\draw[black] (T14)--(B14);
\draw[black] (B1)  .. controls +(.1,.5) and +(-.1,.5) ..  (B2);
\draw[black] (B4)  .. controls +(.1,1) and +(-.1,1) ..  (B9);
\draw[black] (B5)  .. controls +(.1,.5) and +(-.1,.5) ..  (B6);
\draw[black] (B7)  .. controls +(.1,.5) and +(-.1,.5) ..  (B8);
\draw[black] (B12)  .. controls +(.1,.75) and +(-.1,.75) ..  (B13);
\foreach \i in {1,...,14} 
{ \fill (T\i) circle (4pt); \fill (B\i) circle (4pt); }
\end{tikzpicture}\end{array} \\
d^T =   \begin{array}{c}\begin{tikzpicture}[scale=.4,line width=1pt] 
\foreach \i in {1,...,14} 
{ \path (\i,1) coordinate (T\i); \path (\i,-1) coordinate (B\i); } 
\filldraw[fill= black!10,draw=black!10,line width=4pt]  (T1) -- (T14) -- (B14) -- (B1) -- (T1);
\draw[black] (T1)  .. controls +(.1,-.5) and +(-.1,-.5) ..  (T2);
\draw[black] (T4)  .. controls +(.1,-1) and +(-.1,-1) ..  (T9);
\draw[black] (T5)  .. controls +(.1,-.5) and +(-.1,-.5) ..  (T6);
\draw[black] (T7)  .. controls +(.1,-.5) and +(-.1,-.5) ..  (T8);
\draw[black] (T12)  .. controls +(.1,-.75) and +(-.1,-.75) ..  (T13);
\draw[black] (T3) -- (B1);
\draw[black] (T10).. controls +(0,-1.25) and +(0,1.25) .. (B2);
\draw[black] (T11) -- (B9);
\draw[black] (T14) -- (B14);
\draw[black] (B4)  .. controls +(.1,.75) and +(-.1,.75) ..  (B7);
\draw[black] (B5)  .. controls +(.1,.5) and +(-.1,.5) ..  (B6);
\draw[black] (B3)  .. controls +(.1,1) and +(-.1,1) ..  (B8);
\draw[black] (B10)  .. controls +(.1,1) and +(-.1,1) ..  (B13);
\draw[black] (B11)  .. controls +(.1,.5) and +(-.1,.5) ..  (B12);
\foreach \i in {1,...,14} 
{ \fill (T\i) circle (4pt); \fill (B\i) circle (4pt); }
\end{tikzpicture} \end{array}
\end{array} = 
 \begin{array}{c} \begin{tikzpicture}[scale=.4,line width=1pt] 
\foreach \i in {1,...,14} 
{ \path (\i,1) coordinate (T\i); \path (\i,-1) coordinate (B\i); } 
\filldraw[fill= blue!10,draw=blue!10,line width=4pt]  (T1) -- (T14) -- (B14) -- (B1) -- (T1);
\draw[black] (T4)  .. controls +(.1,-.75) and +(-.1,-.75) ..  (T7);
\draw[black] (T5)  .. controls +(.1,-.5) and +(-.1,-.5) ..  (T6);
\draw[black] (T3)  .. controls +(.1,-1) and +(-.1,-1) ..  (T8);
\draw[black] (T10)  .. controls +(.1,-1) and +(-.1,-1) ..  (T13);
\draw[black] (T11)  .. controls +(.1,-.5) and +(-.1,-.5) ..  (T12);
\draw[black] (T1) -- (B1);
\draw[black] (T2) -- (B2);
\draw[black] (T9)--(B9);
\draw[black] (T14)--(B14);
\draw[black] (B4)  .. controls +(.1,.75) and +(-.1,.75) ..  (B7);
\draw[black] (B5)  .. controls +(.1,.5) and +(-.1,.5) ..  (B6);
\draw[black] (B3)  .. controls +(.1,1) and +(-.1,1) ..  (B8);
\draw[black] (B10)  .. controls +(.1,1) and +(-.1,1) ..  (B13);
\draw[black] (B11)  .. controls +(.1,.5) and +(-.1,.5) ..  (B12);
\foreach \i in {1,...,14} 
{ \fill (T\i) circle (4pt); \fill (B\i) circle (4pt); }
\end{tikzpicture}
\end{array} 
$$
\end{enumerate}
\end{examp}

\end{subsection}

\begin{subsection}{A Gelfand model  for $\A_k$}

For any of our diagram algebras $\A_k$, let
\begin{equation}
\begin{array}{lcl}
\I_{\A_k}^{r,f} &=& \{\, d \in \AA_k \, \vert \,  \hbox{ $d$ is symmetric, $\pn(d) = r$, and $d$ has $f$ fixed blocks }\}, \\
\I_{\A_k}^{r} &=& \{\, d \in \AA_k \, \vert \,  \hbox{ $d$ is symmetric, $\pn(d) = r$ }\}, \\
\I_{\A_k} &=& \{\, d \in \AA_k \, \vert \,  \hbox{ $d$ is symmetric }\}.
\end{array}
\end{equation}
For $0 \le f \le r \le k$, define
\begin{equation}
\begin{array}{lcl}
\M_{\A_k}^{r,f} &=& \CC\text{-span}\{\, d \, \vert \, d \in \I_{\A_k}^{r,f} \, \},\\
\end{array}
\end{equation}
where  $\W_{\A_k}^{r,f}= 0$ if $\I_{\A_k}^{r,f} = \emptyset$,
and let
\begin{equation}
\begin{array}{rcl}
\M_{\A_k}^r 
&=& \CC\text{-span}\{\, d \, \vert \, d \in \I_{\A_k}^{r} \, \} ,\\
&=& \displaystyle{\bigoplus_{f=0}^r \M_{\A_k}^{r,f}, }\\
\end{array}
 \qquad\hbox{and}\qquad 
 \begin{array}{rcl}
 \M_{\A_k} 
 &=& \CC\text{-span}\{\, d \, \vert \, d \in \I_{\A_k} \, \},\\
&=& \displaystyle{\bigoplus_{r=0}^k \M_{\A_k}^{r}  = \bigoplus_{r=0}^k \bigoplus_{f=0}^r \M_{\A_k}^{r,f}. }
\end{array}
\end{equation}

If $d \in \AA_k$ and $t \in \I_{\A_k}^{r,f}$, then  $\pn( d \circ t \circ d^T) \le \pn(t)$. When $\pn( d \circ t \circ d^T) = \pn(t)$, the fixed blocks of $t$ have been permuted, and we let $S(d,t)$ be the sign of the permutation of the fixed blocks of $t$.   
For $d\in \AA_k$ and $t \in \I_{\A_k}^{r,f}$, define
\begin{equation}\label{AnAction}
d \cdot t=\begin{cases}
x^{\kappa(d,t)} S(d,t)\, d \circ t \circ d^T, & \text{if $\pn( d \circ t \circ d^T) = \pn(t)$}, \\
0, &  \text{if $\pn( d \circ t \circ d^T) < \pn(t)$}, \\
\end{cases}
\end{equation}
where $\kappa(d,t)$ is the number of blocks \eqref{PartitionAlgebraMultiplication} removed from the middle row in creating $d \circ t$.  We refer to the above action as \emph{signed conjugation} of $t$ by $d$.

\begin{examp}\label{SignedConjugation}
 {\rm (Signed Conjugation)}  In the following example, there is one block removed  in $d \circ t$ yielding a multiple of  $x$. Furthermore, the three fixed blocks of $t$ are permuted as $(B_1, B_2, B_3) \mapsto (B_3,B_2,B_1)$. Hence, the sign is $\sgn(d,t)=-1$. 
 
 $$
 \begin{array}{rccc} 
d\ = &
\begin{array}{c}\begin{tikzpicture}[scale=.4,line width=1pt] 
\foreach \i in {1,...,14} 
{ \path (\i,1.25) coordinate (T\i); \path (\i,-1.25) coordinate (B\i); } 
\filldraw[fill= black!10,draw=black!10,line width=4pt]  (T1) -- (T14) -- (B14) -- (B1) -- (T1);
\draw[black] (T3)  .. controls +(.1,-.75) and +(-.1,-.75) ..  (T2); \draw[black] (T2).. controls +(.1,-.75) and +(-.1,-.75) .. (T1); \draw[black] (T1) --(B1); \draw[black] (B1) .. controls +(.1,.75) and +(-.1,.75) ..  (B2);
\draw[black] (T6)  .. controls +(.1,-.5) and +(-.1,-.5) ..  (T4);
\draw[black] (T4) .. controls +(0,-1.3) and +(0,1.3) ..  (B12);
\draw[black] (T14)  .. controls +(.1,-.5) and +(-.1,-.5) ..  (T12); \draw[black] (T12) .. controls +(.1,-.75) and +(-.1,-.75) .. (T11); \draw[black] (T11).. controls +(.1,-.75) and +(-.1,-.75) ..(T8); \draw[black] (T8)--(B6);
\draw[black] (T7) --(B5);
\draw[black] (T13) .. controls +(0,-1.3) and +(0,-1.3) ..  (T10);
\draw[black] (T9)--(B3);
\draw[black] (B7) .. controls +(.1,.75) and +(-.1,.75) ..(B9); \draw[black] (B9) .. controls +(.1,.75) and +(-.1,.75) ..(B10); \draw[black] (B10) .. controls +(.1,.75) and +(-.1,.75) ..(B11);
\foreach \i in {1,...,14} 
{ \fill[black,draw=black]  (T\i) circle (4pt); \fill[black,draw=black]  (B\i) circle (4pt); }
\end{tikzpicture}\end{array}
&\\
t\ = &\begin{array}{c}\begin{tikzpicture}[scale=.4,line width=1pt] 
\foreach \i in {1,...,14} 
{ \path (\i,1.25) coordinate (T\i); \path (\i,-1.25) coordinate (B\i); } 
\filldraw[fill= blue!10,draw=blue!10,line width=4pt]  (T1) -- (T14) -- (B14) -- (B1) -- (T1);
\draw[black] (T1)  .. controls +(.1,-.75) and +(-.1,-.75) ..  (T2); 
\draw[black] (B1)  .. controls +(.1,.75) and +(-.1,.75) ..  (B2);
\draw[black] (T1) .. controls +(0,-1.3) and +(0,1.3) ..  (B5); 
\draw[black] (B1) .. controls +(0,1.3) and +(0,-1.3) ..  (T5); 

\draw[cyan](T10) .. controls +(.1,-1.25) and +(-.1,-1.25) ..  (T4);
 \draw[cyan] (B4) .. controls +(.1,1.25) and +(-.1,1.25) ..  (B10);
\draw[cyan] (T4) .. controls +(.1,-.75) and +(-.1,-.75) ..  (T3);
 \draw[cyan](T3)--(B3);
 \draw[cyan](B3) .. controls +(.1,.75) and +(-.1,.75) ..   (B4);

\draw[black] (B1)  .. controls +(.1,.75) and +(-.1,.75) .. (B2);

\draw[orange] (T8) .. controls +(.1,-.75) and +(-.1,-.75) ..  (T6); 
\draw[orange] (T6)--(B6); 
\draw[orange] (B6).. controls +(.1,.75) and +(-.1,.75) ..  (B8);

\draw[black] (T9) .. controls +(.1,-.75) and +(-.1,-.75) ..  (T11); 
\draw[black] (B9) .. controls +(.1,.75) and +(-.1,.75) ..   (B11);
\draw[purple] (T14) .. controls +(.1,-.75) and +(-.1,-.75) .. (T12); \draw[purple] (T12)--(B12); \draw[purple] (B12) .. controls +(.1,.75) and +(-.1,.75) ..  (B14);  
\foreach \i in {1,...,14} 
{ \fill[black,draw=black]  (T\i) circle (4pt); \fill[black,draw=black]  (B\i) circle (4pt); }
\end{tikzpicture}\end{array}
&
\!\!\!
= \ \ 
(- x) \begin{array}{c}
\begin{tikzpicture}[scale=.4,line width=1pt] 
\foreach \i in {1,...,14} 
{ \path (\i,1.25) coordinate (T\i); \path (\i,-1.25) coordinate (B\i); } 
\filldraw[fill= blue!10,draw=blue!10,line width=4pt]  (T1) -- (T14) -- (B14) -- (B1) -- (T1);
\draw[black] (T1)  .. controls +(.1,-.75) and +(-.1,-.75) ..  (T2); \draw[black](T2) .. controls +(.1,-.75) and +(-.1,-.75) ..  (T3);
\draw[black] (T1) .. controls +(0,-1.3) and +(0,1.3) ..  (B7); 
\draw[black] (B1) .. controls +(0,1.3) and +(0,-1.3) ..  (T7); 
\draw[black] (B1)  .. controls +(.1,.75) and +(-.1,.75) ..  (B2); \draw[black](B2) .. controls +(.1,.75) and +(-.1,.75) ..  (B3);
\draw[purple] (T4)  .. controls +(.1,-.5) and +(-.1,-.5) ..  (T6);
\draw[purple] (B4)  .. controls +(.1,.5) and +(-.1,.5) ..  (B6);
\draw[purple] (T4)--(B4);
\draw[cyan] (T9)--(B9);
\draw[black] (T13) .. controls +(0,-1.3) and +(0,-1.3) ..  (T10);
\draw[black] (B13) .. controls +(0,1.3) and +(0,1.3) ..  (B10);

\draw[orange] (T14) .. controls +(.1,-.75) and +(-.1,-.75) .. (T12); 
\draw[orange] (T12) .. controls +(.1,-.75) and +(-.1,-.75) .. (T11); 
\draw[orange] (T11) .. controls +(.1,-.75) and +(-.1,-.75) ..(T8);
 \draw[orange] (T8)--(B8); 
 \draw[orange] (B8) .. controls +(.1,.75) and +(-.1,.75) ..  (B11);
  \draw[orange] (B11) .. controls +(.1,.75) and +(-.1,.75) ..(B12); 
  \draw[orange] (B12)  .. controls +(.1,.75) and +(-.1,.75) .. (B14);  
\foreach \i in {1,...,14} 
{ \fill[black,draw=black]  (T\i) circle (4pt); \fill[black,draw=black]  (B\i) circle (4pt); }
\end{tikzpicture}\end{array} \\
d^T\ = & \begin{array}{c}
\begin{tikzpicture}[scale=.4,line width=1pt] 
\foreach \i in {1,...,14} 
{ \path (\i,1.25) coordinate (T\i); \path (\i,-1.25) coordinate (B\i); } 
\filldraw[fill= black!10,draw=black!10,line width=4pt]  (T1) -- (T14) -- (B14) -- (B1) -- (T1);
\draw[black] (B3)  .. controls +(.1,.75) and +(-.1,.75) ..  (B2); \draw[black](B2).. controls +(.1,.75) and +(-.1,.75) .. (B1);\draw[black] (B1) --(T1); \draw[black] (T1) .. controls +(.1,-.75) and +(-.1,-.75) .. (T2);
\draw[black] (B6)  .. controls +(.1,.5) and +(-.1,.5) ..  (B4);
\draw[black] (B4) .. controls +(0,1.3) and +(0,-1.3) ..  (T12);
\draw[black] (B14)  .. controls +(.1,.5) and +(-.1,.5) ..  (B12);\draw[black] (B12) .. controls +(.1,.75) and +(-.1,.75) .. (B11); \draw[black](B11).. controls +(.1,.75) and +(-.1,.75) ..(B8); \draw[black] (B8)--(T6);
\draw[black] (B7) --(T5);
\draw[black] (B10) .. controls +(0,1.3) and +(0,1.3) ..   (B13);
\draw[black] (B9)--(T3);
\draw[black] (T7) .. controls +(.1,-.75) and +(-.1,-.75) ..(T9);\draw[black] (T9) .. controls +(.1,-.75) and +(-.1,-.75) ..(T10); \draw[black](T10) .. controls +(.1,-.75) and +(-.1,-.75) ..(T11);
\foreach \i in {1,...,14} 
{ \fill[black,draw=black]  (T\i) circle (4pt); \fill[black,draw=black]  (B\i) circle (4pt); }
\end{tikzpicture}\end{array}
\end{array}.
$$
\end{examp}

\begin{examp} The signs in Example \ref{ConjugateExample} are (a) $x^2$, (b) $-x^3$, and (c) $x^4$, respectively.
\end{examp}

\begin{rem}  Observe the following subtlety in the definition of this action:  as a product in $\A_k$ we have  $d t d^T = x^{2 \kappa(d,t)} d \circ t \circ d^T$, since each block removed from the middle row in $d\circ t$ has a mirror image in $t \circ d^T$; however, we require  $d \cdot t = x^{\kappa(d,t)} S(d,t)\, d \circ t \circ d^T$ in order to make this an algebra action, as will be seen in the proof of Proposition \ref{prop:Module}.
\end{rem}

\begin{rem}  \label{SkRestriction}
When the action in \eqref{AnAction} is restricted to the symmetric group, we exactly get the action of $\S_k$ on involutions $\I_k$ defined in equation \eqref{SymmetricGroupSignedConjugation}
\end{rem}

\begin{prop}\label{prop:Module} The action defined in \eqref{AnAction} makes $\M_{\A_k}^{r,f}$ an $\A_k$-module.
\end{prop}

\begin{proof}  We  show that $(d_1 d_2) \cdot t = d_1 \cdot (d_2 \cdot t)$.  If $\pn(d \circ t \circ d^T) < \pn(t)$, then by the associativity of diagram multiplication we have $(d_1 d_2) \cdot t  = 0 =  d_1 \cdot (d_2 \cdot t)$. So we assume  that $\pn(d \circ t \circ d^T) = \pn(t)$.  Let $d_1 \circ  d_2 =  d_3$  and let $d_2 \circ t \circ d_2^T = s$ for some $s \in \I_{\A_k}^{r,f}$.  
Then we have,
\begin{align*}
d_1 \cdot (d_2 \cdot t) 
&=  x^{\kappa(d_2,t)} S(d_2,t) d_1 \cdot (d_2 \circ t \circ d_2^T) \\
& =  x^{\kappa(d_1,s)} x^{\kappa(d_2,t)} S(d_1,s) S(d_2,t)( d_1\circ (d_2\circ t \circ d_2^T) \circ d_1^T)  \\
& =  x^{\kappa(d_1,s)} x^{\kappa(d_2,t)} S(d_1,s) S(d_2,t)(( d_1 d_2) \circ t \circ (d_1 \circ d_2)^T)  \\
& =  x^{\kappa(d_1,s)} x^{\kappa(d_2,t)} S(d_1,s) S(d_2,t)(d_3 \circ t \circ d_3^T).
\end{align*}
On the other hand,
 $
(d_1  d_2) \cdot t = x^{\kappa(d_1, d_2)}  d_3 \cdot t = x^{\kappa(d_1, d_2)} x^{\kappa(d_3,t)}   S(d_3,t)  d_3 t d_3^T,
$
so it suffices to show that 
$$
S(d_3, t) = S(d_1, s) S(d_2,t) \qquad\hbox{and}\qquad  x^{\kappa(d_1, d_2)} x^{\kappa(d_3,t)}=x^{\kappa(d_1,s)} x^{\kappa(d_2,t)}.
$$
From the diagram calculus, we have that $\kappa(d_1,d_2) = \kappa(d_1, s)$ and $\kappa(d_3,t) = \kappa(d_2,t)$,  so the second equality follows immediately.  The first equality corresponds to the composition of permutations of  fixed blocks, and the result follows from the symmetric group fact that the sign of a permutation is multiplicative. 
\end{proof}

For $d \in \AA_k$, let $\tau(d)$ be the set partition of $\{1,\ldots, k\}$ given by restricting $d$ to $\{1, \ldots, k\}$ and let $\beta(d)$ be the set partition of $\{1',\ldots, k'\}$ given by restricting $d$ to $\{1', \ldots, k'\}$. Thus if $t$ is symmetric, then $i \leftrightarrow i'$ is a bijection between $\tau(t)$ and $\beta(t)$.
For each $t \in \I_{\A_k}^{r,f}$, let  $p_t \in \AA_k$ be the unique diagram such that
 \begin{enumerate}
 \item[(a)] $\tau(p_t) = \tau(t)$ and $\beta(p_t) = \beta(t)$
\item[(b)] A block of $t$ is propagating if and only if the corresponding block of $p_t$ is an identity block. 
\end{enumerate}
For example, 
$$\begin{array}{lcl} 
t &= & \begin{array}{c}
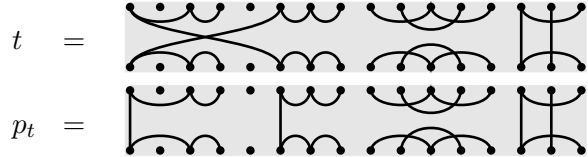
\begin{tikzpicture}[scale=.4,line width=1pt] 
\foreach \i in {1,...,16} 
{ \path (\i,1) coordinate (T\i); \path (\i,-1) coordinate (B\i); } 
\filldraw[fill= black!10,draw=black!10,line width=4pt]  (T1) -- (T16) -- (B16) -- (B1) -- (T1);
\draw[black] (T1) .. controls +(.1,-.65) and +(-.1,-.65) .. (T3);
\draw[black] (T3) .. controls +(.1,-.65) and +(-.1,-.65) .. (T4);
\draw[black] (T6) .. controls +(.1,-.65) and +(-.1,-.65) .. (T7);
\draw[black] (T7) .. controls +(.1,-.65) and +(-.1,-.65) .. (T8);
\draw[black] (B7) .. controls +(.1,.65) and +(-.1,.65) .. (B8);
\draw[black] (T9) .. controls +(.1,-.65) and +(-.1,-.65) .. (T11)  .. controls +(.1,-.65) and +(-.1,-.65) .. (T13);
\draw[black] (T10)  .. controls +(.1,-1) and +(-.1,-1) .. (T12);
\draw[black] (T14) .. controls +(.1,-.65) and +(-.1,-.65) .. (T16);
\draw[black] (T1) .. controls +(0,-1) and +(0,1) .. (B6);
\draw[black] (T6) .. controls +(0,-1) and +(0,1) .. (B1);
\draw[black] (T14)--(B14);
\draw[black] (T15)--(B15);
\draw[black] (B1) .. controls +(.1,.65) and +(-.1,.65) .. (B3);
\draw[black] (B3) .. controls +(.1,.65) and +(-.1,.65) .. (B4);
\draw[black] (B6) .. controls +(.1,.65) and +(-.1,.65) .. (B7);
\draw[black] (B14) .. controls +(.1,.65) and +(-.1,.65) .. (B16);
\draw[black] (B9) .. controls +(.1,.65) and +(-.1,.65) .. (B11)  .. controls +(.1,.65) and +(-.1,.65) .. (B13);
\draw[black] (B10)  .. controls +(.1,1) and +(-.1,1) .. (B12);
\foreach \i in {1,...,16} 
{ \fill (T\i) circle (4pt); \fill (B\i) circle (4pt); } 
\end{tikzpicture}\end{array}
\\
p_t &= & \begin{array}{c}
\begin{tikzpicture}[scale=.4,line width=1pt] 
\foreach \i in {1,...,16} 
{ \path (\i,1) coordinate (T\i); \path (\i,-1) coordinate (B\i); } 
\filldraw[fill= black!10,draw=black!10,line width=4pt]  (T1) -- (T16) -- (B16) -- (B1) -- (T1);
\draw[black] (T1) .. controls +(.1,-.65) and +(-.1,-.65) .. (T3);
\draw[black] (T3) .. controls +(.1,-.65) and +(-.1,-.65) .. (T4);
\draw[black] (T6) .. controls +(.1,-.65) and +(-.1,-.65) .. (T7);
\draw[black] (T7) .. controls +(.1,-.65) and +(-.1,-.65) .. (T8);
\draw[black] (B7) .. controls +(.1,.65) and +(-.1,.65) .. (B8);
\draw[black] (T9) .. controls +(.1,-.65) and +(-.1,-.65) .. (T11)  .. controls +(.1,-.65) and +(-.1,-.65) .. (T13);
\draw[black] (T10)  .. controls +(.1,-1) and +(-.1,-1) .. (T12);
\draw[black] (T14) .. controls +(.1,-.65) and +(-.1,-.65) .. (T16);
\draw[black] (T1) -- (B1);
\draw[black] (T6) -- (B6);
\draw[black] (T14)--(B14);
\draw[black] (T15)--(B15);
\draw[black] (B1) .. controls +(.1,.65) and +(-.1,.65) .. (B3);
\draw[black] (B3) .. controls +(.1,.65) and +(-.1,.65) .. (B4);
\draw[black] (B6) .. controls +(.1,.65) and +(-.1,.65) .. (B7);
\draw[black] (B14) .. controls +(.1,.65) and +(-.1,.65) .. (B16);
\draw[black] (B9) .. controls +(.1,.65) and +(-.1,.65) .. (B11)  .. controls +(.1,.65) and +(-.1,.65) .. (B13);
\draw[black] (B10)  .. controls +(.1,1) and +(-.1,1) .. (B12);
\foreach \i in {1,...,16} 
{ \fill (T\i) circle (4pt); \fill (B\i) circle (4pt); } 
\end{tikzpicture}
\end{array}
\end{array}
$$
Note that $\pn(p_t) = \pn(t)$. 
It follows  from this construction that
\begin{equation} \label{eq:absorption}
p_t t = t p_t = x^\ell t,
\end{equation}
where $\ell$ is the number of non-propagating blocks of $t$.  These diagrams are used in the proof of the following proposition.

\begin{prop} \label{nothingincommon}
If $r\not=s$,  there is no submodule of $\M_{\A_k}^{r,f}$ isomorphic to a submodule of $\M_{\A_k}^{s,g}$. 
\end{prop}

\begin{proof} By Schur's lemma,   $\M_{\A_k}^{r,f}$ and $\M_{\A_k}^{s,g}$ have an isomorphic submodule if and only if there is a nontrivial $\A_k$-module homomorphism $\phi:\M_{\A_k}^{r,f} \rightarrow \M_{\A_k}^{s,g}$.  Assume $r<s$, without loss of generality, and let $t\in \I_{\A_k}^{r,f}$.  Suppose  $\phi: \M_{\A_k}^{r,f} \to \M_{\A_k}^{s,g}$ is a nontrivial $\A_k$-module homomorphism. Then
by \eqref{eq:absorption} and the fact that $\phi$ is an $\A_k$-module homomorphism,
$\phi(t) = \phi\left(x^{-\ell} p_t t\right) = x^{-\ell} p_t \phi(t).$
Now, $\phi(t)$ is a linear combination of symmetric diagrams of rank $s$, but $\pn(p_t) = \pn(t) <s$, and thus by \eqref{AnAction}, $p_t$  acts as 0 on each diagram in the linear combination $\phi(t)$.  Thus $\phi(t) = 0$ for each $t$,  and there are no nontrivial $\A_k$-module homomorphisms. 
\end{proof}

Let  $\varphi_{\A_k}^{r,f}$ be the character of the $\A_k$-module $\M_{\A_k}^{r,f}$. Then $\varphi_{\A_k}^{r} = \sum_{f = 0}^r \varphi_{\A_k}^{r,f}$ is the character of  $\M_{\A_k}^{r}$ and  $\varphi_{\A_k} = \sum_{r = 0}^k \varphi_{\A_k}^{r}$ is the character of  $\M_{\A_k}.$
Let $\varphi_{\C_k}^f$ be the restriction of $\varphi_{\A_k}^{k,f}$ from $\A_k$ to $\C_k$.
Recall from \eqref{CharacterSufficiency} that it is sufficient to compute $\A_k$ characters on  $d \in \C_k$ or $d = a  \e_k$ with $a \in \A_{k'}$. 

\begin{prop} \label{ModelBlockDiagonal}
For each $d \in \AA_k$ and $0 \le f \le r \le k$, we have
$$
\varphi_{\A_k}^{r,f}(d) = 
\begin{cases}
\varphi_{\C_k}^f(d), & \text{if $r=k$ and $\pn(d)=k$,} \\
0, &\text{if $r=k$ and $\pn(d)<k$,} \\
\varphi_{\A_{k'}}^{r,f}(a), & \text{if $r<k$ and $d = a \e_k$ with $a \in \A_{k'}$.} \\
\end{cases}
$$
\end{prop}

\begin{proof} If $r=k$ and  $\pn(d) < k$, then by  \eqref{AnAction} $d$ acts as 0 on every $t \in \I_{\A_k}^{k,f}$ and thus  $\varphi_{\A_k}^{k,f}(d) = 0$.
If $r = k$ and $\pn(d) = k$, then $d \in \C_k$.  The restriction to diagrams of rank $r=k$ is exactly the action of $\C_k$ on $\I_{\A_k}^{k,f} = \I_{\C_k}^f.$  When $\C_k = \CC \S_k$ this is the Saxl representation as observed in Remark  \ref{SkRestriction}.  In the planar case, $\C_k = \CC{\bf 1}_k$, the only planar rank $k$ diagram is ${\bf 1}_k$ and we must have $k = f$.

Let $r < k$ and $d = a \e_k = \e_k a$. Then  $t \in \I_{\A_k}^{r,f}$ contributes to the trace of $d$  only if $d \circ t \circ d^T = t$.  Furthermore,
$d \circ t \circ d^T = (\e_k a) \circ t \circ (\e_ka)^T = \e_k a \circ t \circ a^ T \e_k^T= \e_k a \circ t \circ a^ T \e_k = a' \e_k$ with $a' \in \AA_{k'}$. 
Thus $t$ contributes to the trace only if $t = t' \e_k$ for $t' \in \I_{\A_k'}^{r,f}$.  
Now, $d \cdot t = (a \e_k) \cdot (t' \e_k) 
= x^{\kappa(a \e_k, t' \e_k)} S(a \e_k, t' \e_k) (a  \e_k) \circ (t'  \e_k) \circ (a  \e_k)^T
= x^{\kappa(a \e_k, t' \e_k)} S(a \e_k, t' \e_k) (a \circ t' \circ a^T)  (\e_k \circ \e_k \circ \e_k)
= x^{\kappa(a \e_k, t' \e_k)} S(a \e_k, t' \e_k) (a \circ t' \circ a^T)  \e_k$.  Using the fact that both $a$ and $t'$ commute with $\e_k$, we see that $x^{\kappa(a \e_k, t' \e_k)} = x^{\kappa(a , t')}$ and $S(a \e_k, t' \e_k)  = S(a, t' )$.  Therefore, $t$ contributes to the trace if and only if $t = t' \e_k$ and, in this case, the $t$-$t$ entry of the action of $d$ on $\M_{\A_k}^{r,f}$ equals   the $t'$-$t'$ entry of the action of $a$ on $\M_{\A_{k'}}^{r,f}$.  Thus $\varphi_{\A_k}^{r,f}(d) = \varphi_{\A_{k'}}^{r,f}(a)$.
\end{proof}

When $\A_k$ is nonplanar, we have $\C_r = \CC \S_r$, and the Saxl model (Theorem \ref{thm:Model}) satisfies
\begin{equation}\label{CModelDecomp}
\M_{\C_r} = \bigoplus_{f = 0}^r \M_{\C_r}^f \cong \bigoplus_{f = 0}^r \bigoplus_{\lambda \in \Lambda_{\C_r}^f} \C_r^\lambda \cong \bigoplus_{\lambda \in \Lambda_{\C_r}} \C_r^\lambda, \qquad
\hbox{with $\Lambda_{\C_r} = \bigsqcup_{f=0}^r \Lambda_{\C_r}^f$},
\end{equation}
where $\Lambda_{\C_r}^f$ is the set of partitions of $r$ with $f$ odd parts.  Here  $\M_{\C_r}^f = 0$ and $\Lambda_{\C_r}^f = \emptyset$ if $r - 2f$ is not even.  When $\A_k$ is planar, we have $\C_r = \CC {\bf 1}_r$. In this case, the  model is trivial and satisfies \eqref{CModelDecomp} with $\M_{\C_r}^f = 0$, if $f \not= r$, and $\M_{\C_r}^r = \M_{\C_r}  = \CC {\bf 1}$.  We have $\Lambda_{\C_r}^f = \emptyset$, if $r \not = f$, and $\Lambda_{\C_r}^r = \Lambda_{\C_r}^r = \{(r)\}.$ 

By \eqref{inducctivereplabels}, the irreducible $\A_k$ modules are indexed by
\begin{equation}
\Lambda_{\A_k} = \bigsqcup_{r=0}^k \Lambda_{\C_r} = \bigsqcup_{r=0}^k \bigsqcup_{f=0}^r \Lambda_{\C_r}^f,
\end{equation}
Applying Proposition \ref{nothingincommon} and Proposition \ref{ModelBlockDiagonal} to Theorem \ref{JBCModelCharacters} gives the following theorem.

\begin{thm} \label{ModelRep} {\rm (Model Representation of $\A_k$)} For $k \ge 0$, let $\A_k$ be any  of the diagram algebras defined in  \ref{sec:DiagAlgs}, with $x$ chosen such that $\A_k$ is semisimple.  Let  $\{\A_k^\lambda \vert \lambda \in \Lambda_{\A_k} \}$  denote a complete set of irreducible $\A_k$-modules. Then
for  each $0 \le f \le r \le k$, we have 
\begin{enumerate}
\item[{\rm (a)}] $\displaystyle{\M_{\A_k}^{r,f} \cong \bigoplus_{\lambda \in \Lambda_{\C_r}^f} \A_k^\lambda}$ \qquad\hbox{and}\qquad
 $\displaystyle{\M_{\A_k}^r  = \bigoplus_{f=0}^r  \M_{\A_k}^{r,f} \cong \bigoplus_{\lambda \in \Lambda_{\C_r}} \A_k^\lambda}$,
\item[{\rm (c)}]  $\displaystyle{\M_{\A_k} = \bigoplus_{r=0}^k \bigoplus_{f=0}^r  \M_{\A_k}^{r,f} \cong \bigoplus_{\lambda \in \Lambda_{\A_k}} \A_k^\lambda,}$
\end{enumerate}
where $\M_{\A_k}^{r} =0$ and $\M_{\A_k}^{r,f} =0$ if there do not exist symmetric diagrams  in $\A_k$  of rank $r$ or of rank $r$ with $f$ fixed points.

\end{thm}

\begin{cor} \label{PlanarCorollary} If $\A_k$ is planar, then $\M_{\A_k}^r \cong \A_k^{(r)}$  is irreducible, and  thus $
\M_{\A_k} = \bigoplus_{r=0}^k \M_{\A_k}^{r} 
$
is a decomposition into irreducible $\A_k$-modules.
\end{cor}

\end{subsection}

\end{section}

\begin{section}{Gelfand Models for Diagram Algebras}
\label{ModelDiags}

In this section we illustrate the combinatorics of our  model representation for each diagram algebra. We classify and count the symmetric diagrams of each type according to their rank and number of fixed blocks.  These form a basis for the model representation defined in Theorem \ref{ModelRep}.  As above, we need $\CC$ to be chosen such that the diagram algebra is semisimple.  For example, we may choose $char(\CC) = 0$. In some cases, for example Rui's criterion \cite{Rui} on $char(\CC)$ for the semisimplicity of the Brauer algebra, it is known for which positive characteristics the algebra is semeisimple.

\begin{subsection}{The partition algebra $\P_k(x)$}

The partition algebra $\P_k(x)$ is spanned by the set partitions $\PP_k$ defined in  \ref{sec:PartMonoid} and has dimension equal to the Bell number $B(2k)$.  For $x \in \CC$ such that $x \not \in \{0,1,\ldots, 2k-1\}$, $\P_k(x)$ is semismple (see \cite{MS} or \cite{HR2}) with 
 irreducible modules   indexed by partitions in the set
\begin{equation}
\Lambda_{\P_k} = \{\ \lambda \vdash r\ \vert \ 0 \le r \le k\ \}.
\end{equation}

For each $0 \le \ell \le \lfloor r/2 \rfloor$ there exist symmetric partition algebra diagrams in $\I_{\P_k}^{r,f}$ of rank $r$ with $\ell$ blocks that are transposed (i.e., propagating, nonidentity blocks) and $f=r-2\ell$ fixed blocks. The number of these symmetric diagrams is
\begin{equation}\label{NumberofPartitionSymmetricDiagrams}
\dim  \M_{\P_k}^{r,r-2\ell} =\left\vert \I_{\P_k}^{r,r - 2 \ell} \right\vert = \sum_{b=r}^k \mathbf{S}(k,b)\binom{b}{r}\binom{r}{2\ell}(2\ell-1)!!,
\end{equation}
where $\mathbf{S}(k,b)$ is a Stirling number of the second kind. This sum is justified as follows:  first partition the top and bottom rows of a symmetric diagram identically into $b$ blocks in $\mathbf{S}(k,b)$ ways. Then choose $r$ of these blocks to be propagating, and from those $r$ blocks,  choose $2 \ell$ of them to correspond to transpositions  and  match them up in $(2\ell-1)!!$ ways.  The remaining $r - 2\ell$ blocks are fixed.

The model representation for $\P_k(x)$ satisfies,
\begin{equation}\label{PartitionModelDecomposition}
\M_{\P_k}^{r,f} = \bigoplus_{ \genfrac{}{}{0pt}{}{\lambda \vdash k}{\odd(\lambda)=f}}  \P_k^\lambda \qquad\hbox{and}\qquad
\M_{\P_k} = \bigoplus_{r=0}^k \bigoplus_{\ell=0}^{\lfloor r/2 \rfloor}\M_{\P_k}^{r,r-2\ell} =  \bigoplus_{\lambda \in \Lambda_{\P_k}}  \P_k^\lambda.
\end{equation}
If we let $\p_k = |\I_{\P_k}| = \sum_{r=0}^k \sum_{\ell=0}^{\lfloor r/2\rfloor} |\I_{\P_k}^{r,r-2 \ell}| = \dim \M_{\P_k}$ denote the total number of symmetric diagrams in $\P_k(x)$, then $\p_k$ is the sum of the degrees of the irreducible $\P_k(x)$-modules (which can be found in \cite{Ma}, \cite{HR2}, \cite{Ha1}).  The first few values of $\p_k$ are 1, 2, 7, 31, 164, 999,  6841, 51790, 428131.
The sequence $\mathsf{p}_k$ is \cite{OEIS}  \href{http://oeis.org/A002872}{A002872}, which equals the number of  type-$B$ set partitions (see Remark \ref{TypeB}), and  has exponential generating function
$
e^{(e^{2x}-3)/2 + e^x} = \sum_{k=0}^\infty \mathsf{p}_k \frac{x^k}{k!}.
$
This generating function is justified in \cite{Mo} in formula $6(5')$ (with $p = 2$ in the notation of \cite{Mo}) and in \cite{Qu} (with   $\p_k = H_{2,k}$  in the notation of \cite{Qu}).

\end{subsection}

\begin{subsection}{The Brauer algebra $\B_k(x)$}

The Brauer algebra $\B_k(x)$ is spanned by the Brauer diagrams and
has dimension $\dim \B_k(x) = (2k-1)!!$. For $x \in \CC$ such that  $x \not\in \{ \, x \in \ZZ \, \vert \, 4 - 2k \le x \le k -2 \}$, 
$\B_k(x)$ is semisimple (see \cite{Rui}) with irreducible modules  indexed by partitions in the set
\begin{equation}
\Lambda_{\B_k} = \{\ \lambda \vdash (k - 2 r)\ \vert\ 0 \le r \le \lfloor k/2 \rfloor\ \}.
\end{equation}

Symmetric Brauer diagrams consist of $\ell$ transpositions, $f$ fixed points, and $c$ contractions (symmetric pairs of horizontal edges) with and $f = k - 2c - 2 \ell$.  For example, the symmetric Brauer diagram,
$$
 \begin{array}{c}
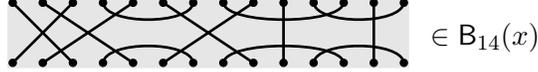
\begin{tikzpicture}[scale=.4,line width=1pt] 
\foreach \i in {1,...,14} 
{ \path (\i,1) coordinate (T\i); \path (\i,-1) coordinate (B\i); } 
\filldraw[fill= black!10,draw=black!10,line width=4pt]  (T1) -- (T14) -- (B14) -- (B1) -- (T1);
\draw[black] (T1) -- (B3);
\draw[black] (T3) -- (B1);
\draw[black] (T2) -- (B5);
\draw[black] (B2)--(T5);
\draw[black] (T6)--(B9);
\draw[black] (T9)--(B6);
\draw[black] (T4)  .. controls +(.1,-.75) and +(-.1,-.75) ..  (T7);
\draw[black] (B4) .. controls +(.1,.75) and +(-.1,.75) .. (B7);
\draw[black] (T11)  .. controls +(.1,-.75) and +(-.1,-.75) ..  (T14);
\draw[black] (B11) .. controls +(.1,.75) and +(-.1,.75) .. (B14);
\draw[black] (T8)  .. controls +(.1,-.75) and +(-.1,-.75) ..  (T12);
\draw[black] (B8) .. controls +(.1,.75) and +(-.1,.75) .. (B12) ;
\draw[black] (T10) -- (B10);
\draw[black] (T13)--(B13);
\foreach \i in {1,...,14} 
{ \fill (T\i) circle (4pt); \fill (B\i) circle (4pt); }
\end{tikzpicture} 
\end{array} \in\B_{14}(x)
$$
has $\ell = 3$ transpositions $(1,3), (2,5),$ $(6,9)$, $c = 3$ contractions in positions $\{4,7\}, \{8,12\}, \{11,14\}$,  $f=2$ fixed points in positions 10 and 13, and rank $r = 8$.  Symmetric Brauer diagrams have rank $r = k -2c$, for $0 \le c \le \lfloor k/2 \rfloor$, and the number of symmetric diagrams of rank $r$ with with $f = r - 2 \ell$ fixed points equals
\begin{equation}\label{NumberOfBrauerSymmetricDiagrams}
\dim  \M_{\B_k}^{r,f} = \dim  \M_{\B_k}^{r,r-2\ell} =\left\vert \I_{\B_k}^{r,r-2\ell} \right\vert = \binom{k}{r}(k-r-1)!!\binom{r}{2\ell}(2\ell-1)!!.
\end{equation}
This count is justified as follows:  choose the $r$ positions for the propagating edges in $\binom{k}{r}$ ways and pair the remaining $k-r$ positions for contractions in $(k-r-1)!!$ ways. Among the propagating edges,  choose $r-2\ell$  fixed points and pair the remaining edges in transpositions in $(2\ell-1)!!$ ways.

The model representation for $\B_k(x)$  satisfies,
\begin{equation}\label{BrauerModelDecomposition}
\M_{\B_k}^{r,f} \cong \bigoplus_{ \genfrac{}{}{0pt}{}{\lambda \vdash r}{{\rm odd}(\lambda) = f}} \B_k^\lambda 
\qquad\hbox{ and } \qquad
\M_{\B_k} \cong \bigoplus_{c=0}^{\lfloor k/2 \rfloor} \bigoplus_{\ell = 0}^{\lfloor (k-2c)/2 \rfloor} \M_{\B_k}^{k-2c,k - 2c-2 \ell} \cong \bigoplus_{\lambda \in \Lambda_{\B_k}} \B_k^\lambda.
\end{equation}
If we let $\b_k = |\I_{\B_k}| = \sum_{c=0}^{\lfloor k/2 \rfloor} \sum_{\ell = 0}^{\lfloor (k-2c)/2 \rfloor} |\I_{\B_k}^{k-2c,k-2c-2\ell}| = \dim \M_{\B_k}$ denote the total number of symmetric diagrams in $\B_k(x)$, then $\b_k$ is the sum of the degrees of the irreducible $\B_k(x)$-modules (see \cite{Ra}).  This value can be obtained by summing  \eqref{NumberOfBrauerSymmetricDiagrams} over the given values of $c$ and $\ell$ or by summing over $m = c + \ell$  as follows,
\begin{equation}
\dim \M_{\B_k} = \sum_{m=0}^{\lfloor k/2 \rfloor} \binom{k}{2m}(2m-1)!! 2^m =\sum_{m=0}^{\lfloor k/2 \rfloor} \binom{k}{2m} \frac{(2m)!}{m!}
=\sum_{m=0}^{\lfloor k/2 \rfloor} \binom{k}{2m} \binom{2m}{m} m!.
\end{equation}
Here we choose $2m$ points to be the endpoints of the transpositions and contractions, we pair them up in $(2m-1)!!$ ways, and then we decide in $2^m$ ways if each is to be a transposition or a contraction.   The first few values of $\b_k$ are 1, 1, 3, 7, 25, 81, 331, 1303, 5937, which is
 \cite{OEIS}  \href{http://oeis.org/A047974}{A047974} with exponential generating function 
$
b(x) = e^{x^2+x} = \sum_{k=0}^\infty \mathsf{b}_k \frac{x^k}{k!}.
$
To justify this generating function,  verify that $\b_k$ satisfies the recurrence  $\b_{k+2} = \b_{k+1} + (2k+1) \b_k$ and therefore $b''(x) = (1 + 2x) b'(x) + 2 b(x)$ which has solution $b(x) = e^{x^2+x}$.

\end{subsection}

\begin{subsection}{The rook monoid algebra $\R_k$}

The  rook monoid algebra $\R_k$ is the subalgebra  spanned by rook monoid diagrams with parameter $x=1$.  It has dimension 
$\dim \R_k = \sum_{\ell = 0}^k \binom{k}{\ell}^2 \ell!$ and is semisimple (see \cite{So}, \cite{Ha2}, \cite{KM})  with irreducible modules labeled by
\begin{equation}
\Lambda_{\R_k} = \{\ \lambda \vdash r \ \vert\ 0 \le r \le  k \ \}.
\end{equation}

Symmetric  rook monoid diagrams in $\R_k$ consist of $f$ fixed points, $\ell$ transpositions, and $k - f - 2 \ell$ vertical pairs of empty vertices.  For example, the symmetric rook monoid diagram,
$$
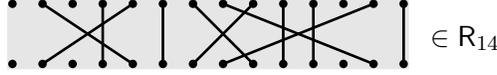
\begin{array}{c} \begin{tikzpicture}[scale=.4,line width=1pt] 
\foreach \i in {1,...,14} 
{ \path (\i,1) coordinate (T\i); \path (\i,-1) coordinate (B\i); } 
\filldraw[fill= black!10,draw=black!10,line width=4pt]  (T1) -- (T14) -- (B14) -- (B1) -- (T1);
\draw[black] (T2) -- (B5);
\draw[black] (T5) -- (B2);
\draw[black] (T4) -- (B4);
\draw[black] (T6) -- (B6);
\draw[black] (T7) -- (B9);
\draw[black] (T9) -- (B7);
\draw[black] (T8) -- (B13);
\draw[black] (T13) -- (B8);
\draw[black] (T10)--(B10);
\draw[black] (T11) -- (B11);
\draw[black] (T14)--(B14);
\foreach \i in {1,...,14} 
{ \fill (T\i) circle (4pt); \fill (B\i) circle (4pt); }
\end{tikzpicture}\end{array} \in \R_{14}
$$
has $\ell = 3$ transpositions $(2,5), (7,9), (8,13)$,  $f = 5$ fixed points 4, 6, 10,11, 14,    empty vertices in positions 1, 3, 12, and rank $r = 11$.  Observe that $f= r- 2 \ell$ and that every pair  $0 \le f \le r \le k$, with $r - f$ even,  is possible. The number of symmetric rook diagrams of rank $r$ with $f = r-2 \ell$ fixed points is
\begin{equation}\label{NumberofRookMonoidSymmetricDiagrams}
\dim  \M_{\R_k}^{r,f} =\dim  \M_{\R_k}^{r,r-2\ell} =\left\vert \I_{\R_k}^{r,r-2\ell} \right\vert = \binom{k}{r}  \binom{r}{2\ell} (2\ell-1)!!.
\end{equation}
To justify this count,
choose the $r$ positions for propagating edges in $\binom{k}{r}$ ways, choose $r-2\ell$ positions for fixed points among these in $\binom{r}{2\ell}$ ways, and pair the remaining propagating edges into transpositions in $(2\ell-1)!!$ ways. 

The model representation for $\R_k$ satisfies,
\begin{equation}\label{RookMonoidModelDecomposition}
\M_{\R_k}^{r,f} \cong \bigoplus_{ \genfrac{}{}{0pt}{}{\lambda \vdash r}{\odd(\lambda)=f}}   \R_k^\lambda
\qquad\hbox{ and } \qquad
\M_{\R_k} \cong \bigoplus_{r=0}^{k} \bigoplus_{\ell=0}^{\lfloor r/2 \rfloor} \M_{\R_k}^{r,r-2 \ell} \cong \bigoplus_{\lambda \in \Lambda_{\R_k}} \R_k^{\lambda}.
\end{equation}
If we let $\r_k = |\I_{\R_k}| = \sum_{r=0}^{k} \sum_{\ell = 0}^{\lfloor r/2 \rfloor} |\I_{\R_k}^{r,r-2\ell}| = \dim \M_{\R_k}$ denote the total number of symmetric diagrams in $\R_k$, then $\r_k$ is the sum of the degrees of the irreducible $\R_k$-modules (which can be found in \cite{So}, \cite{Ha2}).  The first few values of these dimensions are 1, 2, 5, 14, 43, 142, 499, 1850, 7193.
The sequence $\mathsf{r}_k$ gives the number of ``self-inverse partial permutations" and is   \cite{OEIS}   \href{http://oeis.org/A005425}{A005425}.  Furthermore, $\r_k$ is related to the number of involutions $\s_k$ in the symmetric group (see \ref{sec:Saxl}) by the binomial transform $
\mathsf{r}_k = \sum_{i=0}^k \binom{k}{i} \mathsf{s}_i$ and thus (see \cite[(7.75)]{GKP})  has exponential generating function
$
e^x e^{x^2 / 2 + x} =  e^{x^2 / 2 + 2x}  = \sum_{k=0}^\infty \mathsf{r}_k \frac{x^k}{k!}.
$

\begin{rem} The model representation that we construct with our methods here differs from the model for the rook monoid given in \cite{KM} in the same way that the Saxl symmetric group model differs from the one used by Adin, Postnikov, and Roichman \cite{APR}. See  \ref{ModelComparisonSection}.
\end{rem}

\end{subsection}

\begin{subsection}{The rook-Brauer algebra $\RB_k(x)$}

The rook-Brauer algebra $\RB_k(x)$ is  spanned by  rook-Brauer diagrams and has dimension equal to $\sum_{\ell=0}^{k} \binom{2k}{2\ell}(2 \ell-1)!!$ (see \cite{dH} or \cite{MM}).  For all but finitely many  $x \in \CC$ (the exact values have not been determined), $\RB_k(x)$ is semisimple and its irreducible modules are labeled by
\begin{equation}
\Lambda_{\RB_k} = \{\ \lambda \vdash r \ \vert\ 0 \le r \le \lfloor k \rfloor\ \}.
\end{equation}

Symmetric rook-Brauer diagrams in $\RB_k(x)$ consist of $\ell$ transpositions, $f$ fixed points,  $c$ contractions, and $k - 2 \ell - 2 c - f$ vertical pairs of empty vertices.  For example, the symmetric rook-Brauer diagram,
$$
 \begin{array}{c}
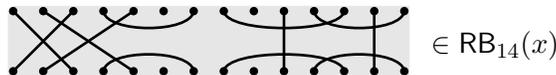
\begin{tikzpicture}[scale=.4,line width=1pt] 
\foreach \i in {1,...,14} 
{ \path (\i,1) coordinate (T\i); \path (\i,-1) coordinate (B\i); } 
\filldraw[fill= black!10,draw=black!10,line width=4pt]  (T1) -- (T14) -- (B14) -- (B1) -- (T1);
\draw[black] (T1) -- (B3);
\draw[black] (T3) -- (B1);
\draw[black] (T2) -- (B5);
\draw[black] (B2)--(T5);           
\draw[black] (T4)  .. controls +(.1,-.75) and +(-.1,-.75) ..  (T7);
\draw[black] (B4) .. controls +(.1,.75) and +(-.1,.75) .. (B7);
\draw[black] (T11)  .. controls +(.1,-.75) and +(-.1,-.75) ..  (T14);
\draw[black] (B11) .. controls +(.1,.75) and +(-.1,.75) .. (B14);
\draw[black] (T8)  .. controls +(.1,-.75) and +(-.1,-.75) ..  (T12);
\draw[black] (B8) .. controls +(.1,.75) and +(-.1,.75) .. (B12) ;
\draw[black] (T10) -- (B10);
\draw[black] (T13)--(B13);
\foreach \i in {1,...,14} 
{ \fill (T\i) circle (4pt); \fill (B\i) circle (4pt); }
\end{tikzpicture} 
\end{array} \in \RB_{14}(x)
$$
has $\ell = 2$ transpositions $(1,3), (2,5),$ $c = 3$ contractions in positions $\{4,7\}, \{8,12\}, \{11,14\}$,  $f=2$ fixed points in positions 10 and 13,  empty vertices in positions 6 and 9, and rank $r =6$.
Observe that these diagrams satisfy $f = r-2 \ell$,  and that every  pair  $0 \le f \le r \le k$, with $r - f$ even,  is possible.  The number of symmetric diagrams of this type is
\begin{equation}\label{NumberOfRookBrauerSymmetricDiagrams}
\dim  \M_{\RB_k}^{r, f}=\dim  \M_{\RB_k}^{r, r-2\ell} = \left\vert \I_{\RB_k}^{r,r-2\ell} \right\vert = \sum_{c=0}^{\lfloor (k-r)/2 \rfloor} \binom{k}{r}\binom{k-r}{2c}(2c-1)!!\binom{r}{2\ell}(2\ell-1)!!,
\end{equation}
where here we sum over the number $c$ of contractions.  This count is justified as follows: in $\binom{k}{r}$ ways,  choose $r$ positions for the propagating edges. Then from the non propagating points,  select the $2c$ endpoints  for the contractions in $\binom{k-r}{2c}$ ways and match them up in $(2c-1)!!$ ways.  Then  choose the $2\ell$ endpoints of the transpositions in $\binom{r}{2\ell}$ ways, and match them up in $(2\ell-1)!!$ ways.

The model representation for $\RB_k(x)$ satisfies,
\begin{equation}\label{RookBrauerModelDecomposition}
\M_{\RB_k}^{r,f} \cong \bigoplus_{ \genfrac{}{}{0pt}{}{\lambda \vdash r}{\odd(\lambda)=f}}  \RB_k^\lambda
\qquad\hbox{ and } \qquad
\M_{\RB_k} \cong \bigoplus_{r=0}^{k} \bigoplus_{\ell=0}^{\lfloor r/2 \rfloor} \M_{\RB_k}^{r,r-2 \ell} \cong \bigoplus_{\lambda \in \Lambda_{\RB_k}} \RB_k^{\lambda}.
\end{equation}

If we let $\rb_k = |\I_{\RB_k}| = \sum_{r=0}^{k} \sum_{\ell = 0}^{\lfloor r/2 \rfloor} |\I_{\RB_k}^{r,r-2\ell}| = \dim \M_{\RB_k}$ denote the total number of symmetric diagrams in $\RB_k(x)$, then $\rb_k$ is the sum of the degrees of the irreducible $\RB_k(x)$-modules (these dimensions can be found in \cite{dH} or \cite{MM}).  
The first few values of $\rb_k$  are 1,  2, 6,  20,  76,  312, 1384, 6512, 32400.
The sequence $\rb_k$ is \cite{OEIS}  \href{http://oeis.org/A000898}{A000898} and it is related to 
the number of symmetric diagrams $\b_k$ in the Brauer algebra  by the binomial transform $
\mathsf{rb}_k = \sum_{i=0}^k \binom{k}{i} \b_i$ and thus (see \cite[(7.75)]{GKP})  has exponential generating function
$
e^x e^{x^2+x}=e^{x^2+2x} = \sum_{k=0}^\infty\mathsf{rb}_k \frac{x^k}{k!}.
$

\end{subsection}

\begin{subsection}{The Temperley-Lieb algebra $\TL_k(x)$}

The Temperley-Lieb algebra $\TL_k(x)$ is spanned by planar Brauer diagrams and has dimension equal to the Catalan number $C_{k} = \frac{1}{k+1}\binom{2k}{k}$.  For $x \in \CC$  that is not the root of $U_k(x/2)$, where $U_k$ is the Chebyshev polynomial of the second kind, $\TL_k(x)$ is  semisimple (see \cite{We} or  \cite{Jo1}) with irreducible modules  indexed by \begin{equation}
\Lambda_{\TL_k} = \{\ k - 2 \ell \ \vert\ 0 \le \ell \le \lfloor k/2 \rfloor \,\}.
\end{equation}

Symmetric Temperley-Lieb diagrams of rank $r$ have $f$ fixed points and $c$ contractions with $r = f = k - 2 c$. For example
$$
\begin{array}{c} \begin{tikzpicture}[scale=.4,line width=1pt] 
\foreach \i in {1,...,14} 
{ \path (\i,1) coordinate (T\i); \path (\i,-1) coordinate (B\i); } 
\filldraw[fill= black!10,draw=black!10,line width=4pt]  (T1) -- (T14) -- (B14) -- (B1) -- (T1);
\draw[black] (T1)  .. controls +(.1,-.5) and +(-.1,-.5) ..  (T2);
\draw[black] (T4)  .. controls +(.1,-1) and +(-.1,-1) ..  (T9);
\draw[black] (T5)  .. controls +(.1,-.5) and +(-.1,-.5) ..  (T6);
\draw[black] (T7)  .. controls +(.1,-.5) and +(-.1,-.5) ..  (T8);
\draw[black] (T12)  .. controls +(.1,-.75) and +(-.1,-.75) ..  (T13);
\draw[black] (T3) -- (B3);
\draw[black] (T10)--(B10);
\draw[black] (T11)--(B11);
\draw[black] (T14)--(B14);
\draw[black] (B1)  .. controls +(.1,.5) and +(-.1,.5) ..  (B2);
\draw[black] (B4)  .. controls +(.1,1) and +(-.1,1) ..  (B9);
\draw[black] (B5)  .. controls +(.1,.5) and +(-.1,.5) ..  (B6);
\draw[black] (B7)  .. controls +(.1,.5) and +(-.1,.5) ..  (B8);
\draw[black] (B12)  .. controls +(.1,.75) and +(-.1,.75) ..  (B13);
\foreach \i in {1,...,14} 
{ \fill (T\i) circle (4pt); \fill (B\i) circle (4pt); }
\end{tikzpicture}\end{array} \in \TL_{14}(x)
$$
has $c = 5$  contractions in positions $\{1,2\}, \{4,9\},$ $\{5,6\},$ $\{7,8\}, \{12,13\}$, 
$f = 4$ fixed points in positions 3, 10, 11, and 14, and rank $r= 4$.
The number of symmetric Temperley-Lieb diagrams (see  \cite[p.\ 545]{We} or \cite[Sec.\ 5.1]{Jo1}) is given by 
\begin{equation}\label{NumberOfTLSymmetricDiagrams}
\dim  \M_{\TL_k}^{r,f} =\left|\I_{\TL_k}^{k - 2 c}\right| = \tldim{k}{c} := \binom{k}{c} - \binom{k}{c-1}.
\end{equation}

The model representation for $\TL_k(x)$ satisfies,
\begin{equation}\label{TLModelDecomposition}
\M_{\TL_k}^{(k - 2 c)} \cong \TL_k^{(k - 2 c)}
\qquad\hbox{ and } \qquad
\M_{\TL_k} \cong \bigoplus_{c=0}^{\lfloor k/2 \rfloor} \M_{\TL_k}^{(k-2c)} \cong \bigoplus_{(k-2 c) \in \Lambda_{\TL_k}} \TL_k^{(k-2c)}.
\end{equation}
The total number of symmetric diagrams in $\TL_k(x)$ is
$$
\tl_k = \dim \M_{\TL_k} = \dim |\I_{\TL_k}| = \sum_{c = 0}^{\lfloor k/2 \rfloor}  |\I_{\TL_k}^{k-2c}|  = \sum_{c = 0}^{\lfloor k/2 \rfloor} \binom{k}{c} - \binom{k}{c-1} =\binom{k}{\lfloor k/2 \rfloor},
$$
which are the central binomial coefficients \cite{OEIS}  \href{http://oeis.org/A000984}{A000984}, and the first few values are 
1, 1, 2, 3, 6, 10, 20, 35, 70.
This sequence has exponential generating function
$
I_0(2x) + I_1(2x) = \sum_{k=0}^\infty\mathsf{tl}_k \frac{x^k}{k!},
$
where $I_n(z)$ is the modified Bessel function of the first kind (see for example \cite[(5.78)]{GKP}).

\begin{rem} The irreducible modules $\TL_k^{(k-2c)}$ are constructed in \cite{We} on ``cup diagrams" (or 1-factors). Cup diagrams correspond exactly to the upper half of a symmetric diagram (since the diagrams are symmetric, only half is needed), and the action of $\TL_k(x)$ on these diagrams is exactly the same as our conjugation action on symmetric diagrams.
\end{rem}

\end{subsection}

\begin{subsection}{The Motzkin algebra $\Motz_k(x)$}

The Motzkin algebra $\Motz_k(x)$ is  spanned by planar rook-Brauer diagrams, which correspond to partial planar matchings of $\{1,\ldots,k,1',\ldots,k'\}$, and so the dimension of $\Motz_k(x)$ is the Motzkin number $M_{2k}$ (see \cite{BH}).  For $x \in \CC$ that is not the root of $U_k((x-1)/2)$, where $U_k$ is the Chebyshev polynomial of the second kind,  $\Motz_k(x)$ is semisimple (see  \cite{BH}) and its irreducible modules are indexed by\begin{equation}
\Lambda_{\Motz_k} = \{ 0, 1, \ldots, k  \}.
\end{equation}

Symmetric Motzkin diagrams of  rank $r$ consist of $f=r$ fixed points,  $c$ contractions, and $k - f - 2c$ pairs of empty vertices.  For example, the symmetric Motzkin diagram,
$$
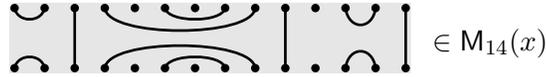
\begin{array}{c} \begin{tikzpicture}[scale=.4,line width=1pt] 
\foreach \i in {1,...,14} 
{ \path (\i,1) coordinate (T\i); \path (\i,-1) coordinate (B\i); } 
\filldraw[fill= black!10,draw=black!10,line width=4pt]  (T1) -- (T14) -- (B14) -- (B1) -- (T1);
\draw[black] (T1)  .. controls +(.1,-.5) and +(-.1,-.5) ..  (T2);
\draw[black] (T4)  .. controls +(.1,-1) and +(-.1,-1) ..  (T9);
\draw[black] (T6)  .. controls +(.1,-.5) and +(-.1,-.5) ..  (T8);
\draw[black] (T12)  .. controls +(.1,-.75) and +(-.1,-.75) ..  (T13);
\draw[black] (T3) -- (B3);
\draw[black] (T10)--(B10);
\draw[black] (T14)--(B14);
\draw[black] (B1)  .. controls +(.1,.5) and +(-.1,.5) ..  (B2);
\draw[black] (B4)  .. controls +(.1,1) and +(-.1,1) ..  (B9);
\draw[black] (B6)  .. controls +(.1,.5) and +(-.1,.5) ..  (B8);
\draw[black] (B12)  .. controls +(.1,.75) and +(-.1,.75) ..  (B13);
\foreach \i in {1,...,14} 
{ \fill (T\i) circle (4pt); \fill (B\i) circle (4pt); }
\end{tikzpicture}\end{array} \in \Motz_{14}(x)
$$
has $c = 4$ contractions in positions $\{1,2\}, \{4,9\}, \{6,8\},$ $\{12,13\}$, $f = 3$ fixed points in positions 3, 10,  14,  vertical pairs of empty vertices in positions 5, 7,  11, and rank $r=3$.
 Observe that every  rank $0 \le r \le k$ is possible.  The number of symmetric diagrams of this type is
\begin{equation}\label{NumberOfMotzkinSymmetricDiagrams}
\dim  \M_{\Motz_k}^{r} =\left\vert \I_{\Motz_k}^{r} \right\vert = \sum_{c = 0}^{\lfloor (k - r)/2 \rfloor} \binom{k}{r + 2 c} \tldim{r+2c}{c},
\end{equation}
where $\left\{{r+2c \atop c} \right\}$ is defined in \eqref{NumberOfTLSymmetricDiagrams}.
This formula is derived in \cite[(3.21)]{BH}.

The model representation for $\M_k(x)$ satisfies,
\begin{equation}\label{MotzkinModelDecomposition}
\M_{\Motz_k}^{r} \cong \Motz_k^{(r)}
\qquad\hbox{ and } \qquad
\M_{\Motz_k} \cong \bigoplus_{r=0}^{k} \M_{\Motz_k}^{r} \cong \bigoplus_{r \in \Lambda_{\Motz_k}} \Motz_k^{(r)}.
\end{equation}
If we let $\m_k = |\I_{\Motz_k}| = \sum_{r=0}^k |\I_{\Motz_k}^r| = \dim \M_{\Motz_k}$ denote the total number of symmetric diagrams in $\M_k(x)$, then $\m_k$ is the sum of the degrees of the irreducible $\Motz_k(x)$-modules.  The first few values of $\m_k$ are $1, 2, 5, 13, 35, 96,  267, 750, 2123, 6046, 17303$.
The sequence $\mathsf{m}_k$ is \cite{OEIS}  \href{http://oeis.org/A005773}{A005773} and it is related to 
the number of symmetric diagrams $\tl_k$ in the Temperley-Lieb algebra by the binomial transform $
\mathsf{m}_k = \sum_{i=0}^k \binom{k}{i} \tl_i$. Thus  (see \cite[(7.75)]{GKP}) $\m_k$ has exponential generating function 
$
e^x (I_0(2x) + I_1(2x)) = \sum_{k=0}^\infty\mathsf{m}_k \frac{x^k}{k!}.
$

\begin{rem} The irreducible modules $\Motz_k^{(r)}$ are constructed in \cite{BH} on 1-factors.  These 1-factors correspond exactly to the upper half of a symmetric Motzkin diagram, and the action of $\Motz_k(x)$ on these diagrams is exactly the same as our conjugation action on symmetric diagrams.  Indeed, it was knowledge of this conjugation action that allowed \cite{BH} to produce the action of $\Motz_k(x)$ on 1-factors.
\end{rem}

\end{subsection}

\begin{subsection}{The planar rook monoid algebra $\PR_k$}

The planar rook monoid algebra $\PR_k$ is spanned by planar rook-monoid diagrams with parameter set to $x=1$.  It has dimension $\binom{2k}{k}$, and is semisimple with irreducible modules labeled by
\begin{equation}
\Lambda_{\PR_k} = \left\{ 0, 1, \ldots, k \right\}.
\end{equation}

Symmetric planar rook monoid diagrams or rank $r$ consist of $f=r$ fixed points and $k -f$ vertical pairs of empty vertices. For example, the symmetric planar rook monoid diagram,
$$
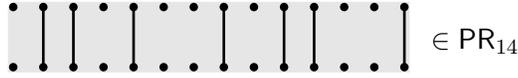
\begin{array}{c} \begin{tikzpicture}[scale=.4,line width=1pt] 
\foreach \i in {1,...,14} 
{ \path (\i,1) coordinate (T\i); \path (\i,-1) coordinate (B\i); } 
\filldraw[fill= black!10,draw=black!10,line width=4pt]  (T1) -- (T14) -- (B14) -- (B1) -- (T1);
\draw[black] (T2) -- (B2);
\draw[black] (T3) -- (B3);
\draw[black] (T5) -- (B5);
\draw[black] (T8) -- (B8);
\draw[black] (T10)--(B10);
\draw[black] (T11) -- (B11);
\draw[black] (T14)--(B14);
\foreach \i in {1,...,14} 
{ \fill (T\i) circle (4pt); \fill (B\i) circle (4pt); }
\end{tikzpicture}\end{array} \in \PR_{14}
$$
has $f = 7$ fixed points in positions 2, 3, 5, 8, 10, 11,  14, and rank $r = 7$. We associate this diagram with its  fixed points $S = \{2,3,5,8,10,11,14\}$, and thus symmetric diagrams correspond exactly to subsets $S \subseteq \{1,2, \ldots, k\}$.  Thus, the number of symmetric diagrams is
\begin{equation}\label{NumberOfPlanarRookSymmetricDiagrams}
\dim  \M_{\PR_k}^{r,f} =\dim  \M_{\PR_k}^{r} =\left\vert \I_{\PR_k}^{r} \right\vert = \binom{k}{r}.
\end{equation}

The model representation for $\PR_k$ satisfies, 
\begin{equation}\label{PlanarRookModelDecomposition}
\M_{\PR_k}^{r} \cong \PR_k^{(r)}
\qquad\hbox{ and } \qquad
\M_{\PR_k} \cong \bigoplus_{r=0}^{k} \M_{\PR_k}^{r} \cong \bigoplus_{r \in \Lambda_{\PR_k}} \PR_k^{(r)}.
\end{equation}
If we let $\pr_k = |\I_{\PR_k}| = \sum_{r=0}^k   |\I_{\PR_k}^r| = \dim \M_{\PR_k}$ denote the total number of symmetric diagrams in $\PR_k$, then $\pr_k$ is the number of subsets of $\{1,2,\ldots,k\}$, so $\pr_k=\dim  \M_{\PR_k} =2^k$  with exponential generating function $e^{2x} = \sum_{k=0}^\infty2^k \frac{x^k}{k!}.$

\begin{rem} The irreducible modules $\PR_k^{(r)}$ are constructed in \cite{FHH} on a basis of $r$-subsets of $\{1,2, \ldots,k\}$.  These $r$-subsets correspond to symmetric rook monoid diagrams, and the action of $\PR_k$ on subsets is exactly the same as our conjugation action on symmetric diagrams.  Indeed, it was knowledge of this conjugation action that led \cite{FHH}  to produce the action of $\PR_k$ on subsets.
\end{rem}

\end{subsection}

\end{section}

\end{document}